\documentclass[10pt]{amsart}
\makeatletter
\let\uppercasenonmath\@gobble
\makeatother
\usepackage{amssymb}
\usepackage[mathscr]{eucal}
\usepackage{graphicx}

\newcommand{\C}{{\mathbb{C}}}

\newcommand{\R}{{\mathbb{R}}}
\newcommand{\Z}{{\mathbb{Z}}}

\newcommand{\cS}{{\mathcal{S}}}

\newcommand{\cC}{{\mathcal{C}}}
\newcommand{\cD}{{\mathcal{D}}}

\newcommand{\Hh}{{\mathcal{H}}}
\newcommand{\cG}{{\mathcal{G}}}
\newcommand{\cK}{{\mathcal{K}}}
\newcommand{\cF}{{\mathcal{F}}}

\newcommand{\cI}{{\mathcal{I}}}

\newcommand{\cE}{{\mathcal{E}}}

\newtheorem{tht}{Theorem}
\newtheorem{thd}{Definition}
\newtheorem{thl}[tht]{Lemma}
\newtheorem{thp}[tht]{Proposition}
\newtheorem{thc}[tht]{Corollary}
\newcommand{\ii}{\mathrm{i}}

\DeclareMathOperator{\sign}{sign}

\DeclareTextFontCommand{\textnf}{\normalfont}

\renewcommand{\setminus}{\smallsetminus}
\renewcommand{\epsilon}{\varepsilon}

\newcommand\http[1]{\href{http://#1}{\nolinkurl{#1}}}

\title{A Resolvent Approach to  the Real Quantum Plane}
\author{Vasyl Ostrovskyi,\, Konrad Schm\"udgen}
\begin{document}
\address{Institute of Mathematics, Ukrainian Academy of Sciences, Tereshchenkivska 3, Kyiv 01601, Ukraine}
\email{vo@imath.kiev.ua}

\address{Mathematisches Institut, Universit\"at Leipzig, Augustusplatz 9/10, 04109 Leipzig, Germany}
\email{schmuedgen@math.uni-leipzig.de}

 \maketitle
 
\begin{abstract}
Let $q\neq \pm 1$ be a complex number of modulus one. This paper deals with the operator relation $AB=qBA$ for self-adjoint operators $A$ and $B$ on a Hilbert space. Two classes of well-behaved representations of this relation are studied in detail and characterized  by  resolvent equations.
\end{abstract}


\textbf{Key  words:} real quantum plane, $q$-commutation relations

\section{Introduction}

The algebraic relation $ab=qba$ is a basic ingredient of the theory  of quantum groups. Let us assume for a moment that this relation holds for a complex number $q$ and some elements $a$ and $b$ of a unital $*$-algebra with involution $x\to x^+$. There are three important cases in which this relation is invariant under the involution. The first  one is when $a$ is unitary (that is, $a^+a=aa^+=1$) and $b$ is hermitian (that is, $b^+=b$), while in the second case we have $a=b^+$. In both cases  $q$ is  real. From an operator-theoretic point of view these two cases are closely related (for instance, by taking the polar decomposition of $a$ in the second case). In the third case $a$ and $b$ are  hermitian and $q$ is of modulus one. All three cases  occur in
the definitions of 
real forms of  quantum groups and quantum algebras, see e.g. 
\cite[Subsections  6.1.7, 9.2.4, 9.2.5]{KS}. The present paper deals with operator representations of the relation $ab=qba$ in this third  case. The corresponding $*$-algebra generated by $a$ and $b$  is  the coordinate algebra of the real quantum plane \cite{S2002} and of the quantum $ax+b$-group \cite{Wo}.

The general operator relation $ab=qba$ has been studied in many papers such as \cite{OS}, \cite{bbp}, \cite{S1992}, \cite{S1994},  \cite{otasz},  \cite{yd}, \cite{mo}, \cite{cho}.

Throughout this paper  
$q$ is a fixed  complex number of {\it modulus one} such that 
$q^2\neq 1$ and  $A$ and $B$ are  {\it self-adjoint} operators on a  Hilbert space $\Hh$. 
We write
\begin{align}\label{qtheta0}
q=e^{-\ii \theta_0},\quad {\rm where}\quad   0< |\theta_0|<\pi.
\end{align}
Our  aim is to study the operator relation 
\begin{equation}\label{qplane1}
 AB=qBA. 
\end{equation}
It turns out that this simple operator relation   leads to 
unexpected  technical difficulties and  
interesting operator-theoretic phenomena. If $A$ and $B$ are bounded and $AB=0$, then (\ref{qplane1}) is obviously satisfied. Let us call representations of (\ref{qplane1}) with $AB=0$ trivial.  Since $q^2\neq 1$, these  are the only representations of (\ref{qplane1}) given by {\it bounded}  operators (see \cite{bbp} or \cite{OS}). Operator representations of algebraic relations have been extensively studied in \cite{OS}, but  the methods developed therein lead only to trivial representations of (\ref{qplane1}). Further, as noted in \cite[p. 1031]{S1992}, in contrast to  Lie algebra relations  the method of analytic vectors fails for the relation (\ref{qplane1}).

Representations of (\ref{qplane1}) by {\it unbounded} self-adjoint operators $A$ and $B$ have been investigated in \cite{S1992} and \cite{S1994}. Some classes
of well-behaved representations of (\ref{qplane1}) have been introduced and classified in \cite{S1994}.
The present paper is devoted to an approach to the operator relation (\ref{qplane1}) that is based on the resolvents of the self-adjoint operators $A$ and $B$. For two classes $\cC_0$ and $\cC_1$ (see Definition   \ref{classc1}) of well-behaved representations  this approach is developed in detail. 

This paper is organized as follows. In Section \ref{prelimrel1} we give a number of reformulations of the operator relation (\ref{qplane1}) in terms of the resolvent $R_\lambda(A)$ and $B$, the resolvent $R_\mu(B)$ and $A$, and the resolvents $R_\lambda (A)$ and $R_\mu (B)$, and we study the largest linear subspace $\cD_q(A,B)$ on which relation (\ref{qplane1}) holds.
In Section \ref{twoclasses} the two classes $\cC_0$ and $\cC_1$ of well-behaved representations of  relation (\ref{qplane1}) are defined and investigated in detail. All irreducible pairs of these classes are built of self-adjoint operators\,  $e^{\alpha Q}$\, and\, $e^{\beta P}$ on the Hilbert space $L^2(\R)$, where\, $Q=x$,\, $P=\ii \frac{d}{dx}$ and\, $\alpha, \beta \in \R$. We prove that the weak resolvent forms $q R_{\lambda q}(A)B\subseteq  B R_{\lambda}(A)$ and $R_{\mu }(B)A\subseteq  qA R_{\mu q}(B)$ of relation (\ref{qplane1}) hold for all pairs $\{A,B\}$ of these classes and for all complex numbers $\lambda$ resp. $\mu$ outside certain critical sectors. Section \ref{characterizations} contains the main results of this paper. These are various theorems which characterize (under additional technical assumptions) well-behaved representations, especially  pairs $\{A,B\}$ of the classes $\cC_0$ and $\cC_1$, by   weak resolvent relations such as $q R_{\lambda q}(A)B\subseteq  B R_{\lambda}(A)$. 
 
Let $A:=e^{\alpha Q}$,\, $B:=e^{\beta P}$,\, and\, $q:=e^{-\ii\alpha\beta}$, where\, $\alpha, \beta \in \R$. As shown in Section \ref{twoclasses},  the resolvent relations \eqref{resolvents1} and \eqref{resolvents1-b} are satisfied on $L^2(\R)$ if $|\alpha\beta|<\pi$ and $\lambda, \mu$ are not in the critical sector    $\cS(q)^+$. From Propositions \ref{prelimn} and \ref{prelimn-b} it follows  that for arbitary numbers $\alpha, \beta, \lambda, \mu$ the relations \eqref{resolvents1} and \eqref{resolvents1-b} holds for vectors of the closures of  subspaces $(B-\mu I)(A-\lambda I)\cD_0$ and
$(A-\lambda q I)(B-\mu q I)\cD_0$, respectively, where $\cD_0={\rm Lin}~\{e^{-\varepsilon x^2+\gamma x}; \varepsilon >0, \gamma\in \C\}$. In  Section \ref{defectspaces} the orthogonal complements of these two subspaces  are explicitely described
and the resolvent actions on these complements are computed.

Some technical preliminaries  are contained in  Section
\ref{operatorpreliminaries}. Amongs these are properties of the operator $e^{\beta P}$ and  a formula for fractional powers of sectorial operators.

Let us collect some basic  notations on operators. Let $T$ be a densely defined closed operator on a Hilbert space. We denote its domain by $\cD(T)$, its resolvent set by $\rho(T)$ and its resolvent $(T-\lambda I)^{-1}$ by $R_\lambda (T)$. Let $U_T$ be the phase operator occuring in the polar decomposition $T=U_T|T|$ of the operator $T$. The symbol $L^2(\R)$ stands for the $L^2$-space with respect to the Lebesgue measure on $\R$.

\section{General considerations on the relation (\ref{qplane1})}\label{prelimrel1}

The following two propositions contain some  simple reformulations of equation  \eqref{qplane1} in terms of the resolvents of the self-adjoint operators $A$ and $B$.

\begin{thp}\label{prelimn}
Suppose that $\lambda$,  $\lambda q\in \rho(A)$ and $\mu$,  $\mu q\in \rho(B)$.

$(i)$~ If $\cD$ is a linear subspace of $  \cD(AB)\cap\cD(BA)$ and \eqref{qplane1} holds for all $f\in \cD$, then 
\begin{align}\label{resolventw}
 B R_{\lambda}(A)g = q R_{\lambda q} (A)Bg
\end{align}
for all\, $g\in \cE:=(A - \lambda  I)\cD$ and   $\cE$ is a  linear subspace of $\cD(B)$.

$(ii)$ If $\cE$ is a linear subspace of $\cD(B)$ such that $R_{\lambda}(A)g\in \cD(B)$ and \eqref{resolventw} is satisfied for all $g\in \cE$, then \eqref{qplane1} holds for all $f\in \cD:=R_\lambda(A)\cE$.

$(iii)$~ If $\cE$ is a linear subspace of $\cD(B)$ and \eqref{resolventw} holds for all $g\in \cE$, then 
\begin{align}\label{resolvents}
 R_{\lambda }(A)R_{\mu}(B)h= qR_{\mu q}(B)R_{\lambda q} (A)h + 
{\mu\lambda q(q- 1)} R_{\mu q}(B)R_{\lambda q} (A)R_{\lambda }(A)R_{\mu }(B)h
\end{align}
for all\,  $h\in\cF:=(B-\mu I)\cE.$ 

$(iv)$~ If $\cF$ is a linear subspace of $\Hh$ such that \eqref{resolvents} holds for all $h\in \cF$, then \eqref{resolventw} is fulfilled for all $g\in \cE:=R_{\mu}(B)\cF.$
\end{thp}
\begin{proof}
(i): Clearly, \eqref{qplane1} implies that
\[
 (A-\lambda q I)Bg = q B(A -\lambda  I)g
\]
for  $f\in \cD$. Hence, for all vectors  of the form $g=(A - \lambda  I)f$, where $f\in \cD$, we have $R_{\lambda }(A)g \in \cD(B)$ and
\[
 R_{\lambda q}(A)(A-\lambda q I) B R_{\lambda}(A)g = 
 q R_{\lambda q} (A)B (A-\lambda I)R_{\lambda }(A)g,
\]
so that
\[
 B R_{\lambda }(A)g = q R_{\lambda q} (A)Bg.
\]

(iii): Let $g\in \cE$. From equation \eqref{resolventw} we obtain 
\begin{align}\label{auxilaresolv}
 (B-\mu q I)R_{\lambda}(A)g& =(BR_{\lambda }(A)-\mu q R_{\lambda }(A))g = 
 (q R_{\lambda q}(A)B-\mu q R_{\lambda }(A))g\nonumber
\\
&=(q R_{\lambda q} (A)(B-\mu  I) + \mu q R_{\lambda q}(A)-\mu q R_{ \lambda }(A))g\nonumber
\\
&=(q R_{\lambda q} (A)(B-\mu  I) + \mu q(\lambda q- 
\lambda ) R_{\lambda q}(A)R_{\lambda }(A))g.
\end{align}
Setting $h=(B-\mu I)g$, we have $g=R_{\mu}(B)h$. Inserting this into \eqref{auxilaresolv} and applying $R_{\mu q}(B)$ to both sides yields
 \eqref{resolvents} for $h\in(B-\mu I)\cE$.
 
 (ii) and (iv) follow  by reversing  the preceding arguments of proofs of (i) and (iii), respectively.
\end{proof}

Using the equalities $R_\lambda(\overline{q} A)=qR_{\lambda q}( A)$ and $R_\mu(\overline{q} B)=qR_{\mu q}( B)$ one can  rewrite \eqref{resolventw} in the form 
\begin{align*}
BR_\lambda(A)g= R_\lambda(\overline{q} A) Bg
\end{align*}
and \eqref{resolvents} as
\begin{align*}
 qR_{\lambda}(A)R_{\mu}(B)h= R_\mu(\overline{q} B)R_\lambda (\overline{q} A)h + 
\mu R_\mu(\overline{q} B)(\overline{q} R_\lambda (\overline{q} A) - R_{\lambda }(A))R_{\mu}(B)h. 
\end{align*}

In a similar manner the following proposition is derived.

\begin{thp}\label{prelimn-b}
Suppose that $\lambda$,  $\lambda q\in \rho(A)$ and $\mu$,  $\mu q\in \rho(B).$ 

$(i)$ ~If equation \eqref{qplane1} is satisfied  for all $f$ of a~linear subspace $ \cD\subseteq   \cD(AB)\cap\cD(BA)$, then 
\begin{align}\label{resolventw-b}
  R_{\mu}(B)Ag = qA R_{\mu q} (B)g
\end{align}
for all\, $g\in \cE:=(B - \mu q  I)\cD$ and $\cE$ is a subspace of $ \cD(A)$.

$(ii)$ If $\cE$ is a linear subspace of $\cD(A)$ such that $R_{\mu q}(B)g\in \cD(A)$ and \eqref{resolventw-b} holds for all $g\in \cE$, then \eqref{qplane1} is true for all $f\in \cD:=R_{\mu q}(B)\cE$.

$(iii)$~ If $\cE$ is a linear subspace of $\cD(A)$ and \eqref{resolventw-b} is satisfied for all $g\in \cE$, then 
\begin{align}\label{resolvents-b}
 R_{\lambda }(A)R_{\mu}(B)h= qR_{\mu q}(B)R_{\lambda q} (A)h + 
{\mu\lambda q(q- 1)} R_{\lambda }(A)R_{\mu }(B)R_{\mu q}(B) R_{\lambda q}(A)h 
\end{align}
for all\,  $h\in\cF:=(A-\lambda q I)\cE.$ 

$(iv)$ If equation \eqref{resolvents-b} holds for all $h$ of a linear subspace $\cF\subseteq \Hh$, then \eqref{resolventw-b} is satisfied for all $g\in \cE:=R_{\lambda q}(A)\cF.$
\end{thp}

Comparing Propositions \ref{prelimn} and \ref{prelimn-b}, especially formulas (\ref{resolvents}) and (\ref{resolvents-b}), we obtain

\begin{thc}\label{commute_ond}
Let $\lambda$,  $\lambda q\in \rho(A)$ and $\mu$,  $\mu q\in \rho(B)$. If equation \eqref{qplane1}  holds on  a linear subspace   $\cD\subseteq\cD(AB)\cap\cD(BA)$,  
then the operators $R_{\lambda }(A)R_{\mu }(B)$ and $R_{\mu q}(B) R_{\lambda q }(A)$ commute on  the linear space $(A-\lambda q I)(B - q\mu  I)\cD \cap (B - \mu  I)(A-\lambda  I)\cD$.
\end{thc}

Without further assumptions the linear subspace $(A - \lambda  I)\cD$ of $\cD(B)$ is neither a core for $B$ nor the subspace  $(B-\mu I)\cE$ is dense in $\Hh$. Note that \eqref{qplane1} for all $f\in \cD$ implies that \eqref{resolvents} holds for all vectors $h\in(B-\mu I)(A-\lambda  I)\cD$ and (\ref{resolvents-b}) is valid for $h\in(A-\lambda q I)(B-\mu q I)\cD.$
\begin{thd}\label{defsete1}
~~$\cD_q(A,B):=\{f\in \cD(BA)\cap \cD(AB):~~ABf=qBAf \}.$
\end{thd} 
Obviously, $\cD_q(A,B)$ is the largest linear subspace of $\Hh$ on which relation \eqref{qplane1} holds. Of course, for arbitary self-adjoint operators $A$ and $B$ it may happen $\cD_q(A,B)=\{0\}$. From Proposition  \ref{prelimn} we immediately obtain the following descriptions of the space $\cD_q(A,B)$:
\begin{align*}
\cD_q(A,B)=R_\lambda (A)\, \{ g\in \cD(B): R_{\lambda}(A)&g\in \cD(B)~~{\rm and}~~B R_{\lambda}(A)g = q R_{\lambda q} (A)Bg\}\\
=R_\lambda (A)R_\mu (B)\, \{h\in \Hh: R_{\lambda }(A)&R_{\mu}(B)h= qR_{\mu q}(B)R_{\lambda q} (A)h\\& + 
{\mu\lambda q(q- 1)} R_{\mu q}(B)R_{\lambda q} (A)R_{\lambda }(A)R_{\mu }(B)h\}.
\end{align*}
Similarly,  Proposition  \ref{prelimn-b} leads to the following descriptions of $\cD_q(A,B)$:
\begin{align*}
\cD_q(A,B)&=R_{\mu q} (B)\, \{ g\in \cD(A): R_{\mu q}(B)g\in \cD(A)~~{\rm and}~~R_{\mu}(B)Ag = q AR_{\mu q} (B)g\}
\\
&=R_\lambda (A)R_\mu (B)\, \bigl\{h\in \Hh: R_{\lambda }(A)R_{\mu}(B)h= qR_{\mu q}(B)R_{\lambda q} (A)h
\\
& \qquad\qquad\qquad\qquad\qquad{}+ 
{\mu\lambda q(q- 1)} (A)R_{\lambda }(A)R_{\mu }(B)R_{\mu q}(B)R_{\lambda q} h \bigr\}.
\end{align*}
In particular, we have
\[
 \cD_q(A,B) \subset R_\lambda(A) \cD(B) \cap R_{\mu q} (B) \cD(A).
\]

The operator relation (\ref{qplane1}) is obviously equivalent to the the relation
\begin{equation}\label{qplane2}
 BAf=\overline{q}ABf.
\end{equation}
Hence $\cD_q(A,B)=\cD_{\overline{q}}(B,A)$.
\medskip

If equation \eqref{resolventw} holds for all vectors $g$ of the whole domain $\cD(B)$, that is, if
\begin{align}\label{resolventw1}
q R_{\lambda q}(A)B\subseteq  B R_{\lambda}(A), 
\end{align}
 we shall say that relation (\ref{resolventw1}) is  the {\bf weak $A$-resolvent form} of equation \eqref{qplane1} for $\lambda$, $\lambda q \in \rho(A)$. 
If equation \eqref{resolventw-b} holds for all vectors $g$ of the  domain $\cD(B)$, that is, if
\begin{align}\label{resolventw1-b}
R_{\mu }(B)A\subseteq  qA R_{\mu q}(B), 
\end{align}
we say that relation (\ref{resolventw1-b}) is  the {\bf weak $B$-resolvent form} of equation \eqref{qplane1} for $\mu$, $\mu q \in \rho(B)$. 
Setting 
$\nu=\mu q$ relation (\ref{resolventw1-b}) can be rewritten as 
\begin{align}\label{resolventw1-bn}
\overline{q} R_{\nu \overline{q}}(B)A\subseteq AR_\nu(B). 
\end{align}
The form (\ref{resolventw1-bn})  of the weak $B$-resolvent relation of \eqref{qplane1}   corresponds to the weak $A$-resolvent form of equation \eqref{qplane2} which is obtained by interchanging $A$ and $B$ and replacing $q$ by $\overline{q}$ .

Further, if equation \eqref{resolvents} is satisfied for all $h\in \Hh$, that is, if 
\begin{align}\label{resolvents1}
 R_{\lambda }(A)R_{\mu}(B)= qR_{\mu q}(B)R_{\lambda q} (A) + 
{\mu\lambda q(q- 1)} R_{ \mu q}(B)R_{\lambda q} (A)R_{\lambda}(A)R_{\mu }(B), 
\end{align}
then equation (\ref{resolvents}) is called the $(A,B)$-{\bf resolvent form} of equation \eqref{qplane1} for  $\lambda$, $\lambda q\in \rho(A)$ and $\mu$, $\mu q\in \rho(B)$. Likewise,
if equation \eqref{resolvents-b} holds for all $h\in \Hh$, that is, if 
\begin{align}\label{resolvents1-b}
 R_{\lambda }(A)R_{\mu}(B)= qR_{\mu q}(B)R_{\lambda q} (A) + 
{\mu\lambda q(q- 1)} R_{\lambda}(A)R_{\mu }(B)R_{ \mu q}(B)R_{\lambda q} (A), 
\end{align}
then equation (\ref{resolvents-b}) is called the $(B,A)$-{\bf resolvent form} of equation \eqref{qplane1} for  $\lambda$, $\lambda q\in \rho(A)$ and $\mu$, $\mu q\in \rho(B)$.

The resolvent relations \eqref{resolvents1} and \eqref{resolvents1-b} can be rewritten as
\begin{align*}
\Bigl(R_{ \mu q}(B)R_{\lambda q} (A) - 
\frac{1}{\mu\lambda q(q- 1)}I \Bigr) \Bigl(R_{\lambda}(A)R_{\mu }(B) +\frac{1}{\mu\lambda (q- 1)} I \Bigr)&=-\frac{1}{\lambda^2\mu^2 q(q- 1)^2}I, 
\\
 \Bigl( R_{\lambda}(A)R_{\mu }(B)+\frac{1}{\mu\lambda (q- 1)}I \Bigr) \Bigl(R_{ \mu q}(B)R_{\lambda q} (A)-\frac{1}{\mu\lambda q(q- 1)}I \Bigr)&=-\frac{1}{\lambda^2\mu^2 q(q- 1)^2} I, 
\end{align*}
respectively. They hold for all vectors from the subspaces $(B-\mu I)(A-\lambda I)\cD_q(A,B)$ and $(A-\lambda q I)(B-\mu qI)\cD_q(A,B)$, respectively.  In Section~\ref{defectspaces} we derive for a class of representations of $\eqref{qplane1}$ the form of resolvent relations on the complements of these subspaces.

\begin{thp}\label{equivalenceweakstrong}
The weak $A$-resolvent form \eqref{resolventw1} is equivalent to  the $(A,B)$-resolvent form \eqref{resolvents1} of equation \eqref{qplane1}. The weak $B$-resolvent form \eqref{resolventw1-b} and the $(B,A)$-resolvent form \eqref{resolvents1-b} of  \eqref{qplane1} are equivalent.
\end{thp}

\begin{proof}
First suppose that \eqref{resolventw1} holds. This means that \eqref{resolventw} is satisfied for all vectors $g\in \cD(B)$. Therefore,  by Proposition  \ref{prelimn}(ii), equation \eqref{resolventw} holds for $h\in (B-\mu I)\cD(B)$. Since $\mu \in \rho(B)$, $(B-\mu I)\cD(B)$ is equal to $\Hh$ which yields \eqref{resolvents1}.

Conversely, assume that \eqref{resolvents1} is fulfilled. Let $g\in \cD(B)$. We set $h= (B-\mu\overline q I)g$ in \eqref{resolvents}. Since the ranges of resolvents of $B$ are contained in the domain of $B$,  the vector in \eqref{resolvents} is in $\cD(B)$, so  we can apply the operator $B-\mu I$ to both sides of \eqref{resolvents}. Then  we obtain \eqref{resolventw} which proves  \eqref{resolventw1}. 

The equivalence of \eqref{resolvents1-b} and \eqref{resolventw1-b} follows by a similar reasoning.
\end{proof}

The next proposition collects a number of basic facts concerning the weak resolvent identities.
\begin{thp}\label{rldb} Suppose that\, $\lambda, \lambda q \in \rho(A)$\, and\, $\mu,\mu q\in \rho(B)$. \\
$(i)$\, $qR_{\lambda q}(A)B \subseteq BR_\lambda (A)$\, if and only if\,  $\cD_q(A,B) = R_\lambda(A) \cD(B)$.\\ 
$(ii)$\, $R_{\mu }(B)A\subseteq  qA R_{\mu q}(B)$\, if and only if\,  $\cD_q(A,B) = R_{\mu q} (B) \cD(A)$.\\
$(iii)$\, $qR_{\lambda q}(A)B \subseteq BR_\lambda (A)$\,if and only if\, $qR_{\overline{\lambda}}(A)B\subseteq BR_{\overline{\lambda q}} (A)$.\\
$(iv)$\, $R_{\mu }(B)A\subseteq  qA R_{\mu q}(B)$\, if and only if\,  $R_{\overline{\mu q}}(B)A\subseteq  qA R_{\overline{\mu}}(B).$\\
$(v)$\, If $qR_{\lambda q}(A)B \subseteq BR_\lambda (A)$, then $\cD_q(A,B)$ is a core for $A$. \\
$(vi)$\, If $R_{\mu }(B)A\subseteq  qA R_{\mu q}(B)$\, then $\cD_q(A,B)$ is a core for $B$.
\end{thp}

\begin{proof}
We carry out the proofs of (i), (iii), and (v). The proofs of (ii), (iv), and (vi) follows by a similar reasoning. 

(i): Throughout this proof let us set $\cD_B:=R_{\lambda} (A) \cD(B).$

First suppose that $qR_{\lambda}(A)B \subseteq BR_\lambda (A)$.  Obviously, $\cD_B\subset \cD(A)$. The inclusion $\cD_B\subset \cD(B)$ follows from $qR_{\lambda}(A)Bf = BR_\lambda (A)f$, $f\in \cD(B)$. Further, we have $B\cD_B\subset \cD(A)$ since $qR_{\lambda}(A)Bf = BR_\lambda (A)f$ and $R_{\lambda}(A)Bf \subseteq \cD(A)$, $f\in \cD(B)$. Also, $A\cD_B\subset \cD(B)$, since $A-\lambda I$ maps $\cD_B$ onto $\cD(B)$. Therefore, $\cD_B\subset \cD_q(A,B)$. Since $g=(A-\lambda I)f \in \cD(B)$ for any $f\in \cD_q(A,B)$, we see that $f=R_\lambda(A)g$, so that $ \cD_q(A,B)\subseteq \cD_B$. Thus, $\cD_B = \cD_q(A,B)$.

Conversely, assume that $\cD_B = \cD_q(A,B)$. Then $(A-\lambda I)\cD_q(A,B)=\cD(B)$ and by Proposition~\ref{prelimn}(i), we have $qR_{\lambda}(A)Bf = BR_\lambda (A)f$ for all $f\in \cD(B)$.

(iii): Suppose that\, $qR_{\lambda q}(A)B \subseteq BR_\lambda (A)$. Since $R_{\lambda q}(A)$ is bounded, we have $(R_{\lambda q}(A)B)^*=B^* (R_{\lambda q}(A))^*=BR_{\overline{\lambda q}} (A)$ and hence  $$\overline{q} \, BR_{\overline{\lambda q}} (A)=(qR_{\lambda q}(A)B)^* \supseteq (BR_\lambda (A))^* \supseteq R_{\overline{\lambda}} (A)B,$$
so that\, $qR_{\overline{\lambda}}(A)B\subseteq BR_{\overline{\lambda q}} (A)$.

The converse direction follows by applying the same implication once again.

(v): Since $\cD_q(A,B)=R_\lambda(A) \cD(B)$ by (i),  $(A-\lambda I)\cD_q(A,B)=\cD(B)$ is dense in $\Hh$. Hence $\cD_q(A,B)$ is a core for $A$. \end{proof}

An immediate consequence of  Proposition~\ref{rldb} is  the following corollary. 
\begin{thc} Let $\lambda, \lambda q \in \rho(A)$ and $\mu,\mu q\in \rho(B)$.
Assume that\, $qR_{\lambda}(A)B \subseteq BR_\lambda (A)$ and $R_{\mu }(B)A\subseteq  qA R_{\mu q}(B)$. Then
\begin{align}\label{dqabcore}
\quad\cD_q(A,B)=R_\lambda(A) \cD(B)=R_{\mu q} (B) \cD(A)
\end{align}
and $\cD_q(A,B)$ is a core for $A$ and $B$.
\end{thc}

\begin{thc}\label{da2core} Suppose that  $\mu,\mu q,\mu q^2\in \rho(B)$. If\, $R_{\mu }(B)A\subseteq  qA R_{\mu q}(B)$\, and\, $R_{\mu q}(B)A\subseteq  qA R_{\mu q^2}(B)$,  then $\cD_{q^2}(A^2,B)$ is a core for $B$.
\end{thc}
\begin{proof}
From the assumptions we derive $R_{\mu }(B)A^2\subseteq  qA R_{\mu q}(B)A \subseteq  q^2A^2 R_{\mu q^2}(B),$ that is, the weak $B$-resolvent form for the relation $A^2B=q^2BA^2$ is satisfied. Therefore, $\cD_{q^2}(A^2,B)$ is a core for $B$ by Proposition \ref{rldb}(vi). 
\end{proof}

The next proposition shows how the resolvent relations \eqref{resolvents1} and \eqref{resolvents1-b} follow from the essential self-adjointness of a certain symmetric operator. 

Let us fix $a$, $b \in \mathbb R$ and choose the branch of the square root  such that 
\begin{align}
 \overline q^{1/2} =\overline { q^{1/2}}\, .
 \end{align}
We define an operator $T$ with domain $\cD(T):= \cD_q(A,B)$ by
\begin{equation}\label{operator_T}
 Tf= \bar q^{1/2} (A-a  q^{1/2})(B-b  q^{1/2})f + \frac{\bar q^{1/2} - q^{1/2}}2  ~ab f,~f\in \cD(T).
\end{equation}
\begin{thl}
The operator $T$  is symmetric.
\end{thl}
\begin{proof}
Clearly, $Tf=\bigl( \bar q^{1/2} AB - b A -a B + \frac{q^{1/2} +\bar q^{1/2}}2\, ab\bigr) f$. Using this formula 
 we derive
\begin{align*}
\langle Tf,g\rangle & =
\langle \bigl( \bar q^{1/2} AB - b A -a B + \frac{q^{1/2} +\bar q^{1/2}}2 \, ab\, \bigr)f, g \rangle\\
&= \langle f,\bigl(  q^{1/2} BA - b A -a B +  \frac{q^{1/2} + \bar q^{1/2}}2\, ab\, \bigr) g\rangle
\\
&=\langle f, \bigl( \bar q^{1/2} AB - b A -a B +  \frac{q^{1/2} +\bar q^{1/2}}2\, ab\, \bigr) g\rangle =\langle Tf, g\rangle
\end{align*}
for $f,g\in \cD(T)$, that is, $T$ is symmetric.
\end{proof}

\begin{thp}
 Assume that $ab\ne0$ and $q^2\ne1$. If the operator $T$ is essentially self-adjoint, then  both resolvent relations \eqref{resolvents1} and \eqref{resolvents1-b} hold on $\Hh$ for $\lambda=a\bar q^{1/2}$,  $\mu=b\bar q^{1/2}$ and the operator $R_{b  q^{1/2}}(B) R_{a  q^{1/2} }(A)$ is normal.
\end{thp}

\begin{proof}
Setting
\[
 \tau = \frac{\bar q^{1/2} - q^{1/2}}2\, ab,
\]
the operator $T$ can be rewritten as
\[
 Tf= \bar q^{1/2} (A-a  q^{1/2})(B-b  q^{1/2})f + \tau  f = q^{1/2} (B-b \bar q^{1/2})(A-a \bar q^{1/2})f - \tau  f
\]
for $f\in \cD(T)$.
Therefore, since\, $T$\, is essentially self-adjoint and $\tau $ is purely imaginary and nonzero (by the assumptions $ab\neq 0$ and $q^2\neq 1$), the set 
\[
\cF_0:=(T-\tau I)\cD(T)= (A-a  q^{1/2} )(B-b q^{1/2})\cD(T)
\]
is dense in $\Hh$. By Proposition \ref{prelimn-b},(i) and (iii), equation (\ref{resolvents-b}) is satisfied for $\lambda=a  \bar q^{1/2}$, $\mu= b \bar q^{1/2}$ and all vectors $h\in \cF_0$. Since\, $\cF_0$\, is dense and all resolvent operators are bounded,  equation (\ref{resolvents-b}) holds for all $h\in \Hh$. That is, we have  
\begin{align}\label{resolone}
 R_{a \bar q^{1/2}}(A)R_{b \bar q^{1/2}}(B)&= qR_{b q^{1/2}}(B)R_{a q^{1/2}}(A) \notag
\\
&\quad{}+ 
{ab(q-1)} R_{a \bar q^{1/2}}(A)R_{b \bar q^{1/2}}(B)R_{b q^{1/2}}(B)R_{a q^{1/2}}(A).
\end{align}
Thus, the  $(B,A)$-resolvent relation \eqref{resolvents1-b} is satisfied.

Similarly, 
we conclude that 
\[
\cF_1:=(T+\tau I)\cD(T)= (B-b \bar q^{1/2})(A-a  \bar q^{1/2}) \cD(T)
\]
is dense in $H$ and equation (\ref{resolvents}) holds for  $\lambda=a  \bar q^{1/2}$, $\mu= b \bar q^{1/2}$ and  $h\in\cF_1$ by Proposition \ref{prelimn},(i) and (iii), and hence for all vectors $h\in \Hh$. That is, the  $(A,B)$-resolvent relation \eqref{resolvents1} is valid and we have 
\begin{align}\label{resoltwo}
 R_{a \bar q^{1/2}}(A)R_{b \bar q^{1/2}}(B)&= qR_{b q^{1/2}}(B)R_{a q^{1/2}}(A) \notag
\\
&\quad{}+ 
{ab(q-1)} R_{b q^{1/2}}(B)R_{a q^{1/2}}(A)R_{a \bar q^{1/2}}(A)R_{b \bar q^{1/2}}(B).
\end{align}
Comparing (\ref{resolone}) and (\ref{resoltwo}) we conclude that $$ R_{a \bar q^{1/2}}(A)R_{b \bar q^{1/2}}(B)R_{b q^{1/2}}(B)R_{a q^{1/2}}(A)=R_{b q^{1/2}}(B)R_{a q^{1/2}}(A)R_{a \bar q^{1/2}}(A)R_{b \bar q^{1/2}}(B)$$ which means that the operator $R_{b  q^{1/2}}(B) R_{a  q^{1/2} }(A)$ is normal.
\end{proof}

\section{Operator-theoretic preliminaries}\label{operatorpreliminaries}

We denote by $P=\ii \frac{d}{dx}$  the momentum operator and by $Q=x$ the position operator acting on the Hilbert space $L^2(\R)$ with respect to the Lebesgue measure on $\R$. Fix $\beta >0$.

\begin{thl}\label{ebatepcondition}
$(i)$ Suppose that $f(z)$ is a holomorphic function on the strip $\cI_\beta:=\{z\in \C: 0 < {\rm Im} z< \beta\}$ such that
\begin{align}\label{ebatepcondition1}
\sup\limits_{0<y<\beta}~ \int_{-\infty}^{+\infty} |f(x+\ii y)|^2 ~dx <\infty.
\end{align}
Set $f_y(x):=f(x+\ii y)$. Then the limits 
$f_0:=\lim_{y \downarrow 0} f_y(x)$ and $f_\beta:=\lim_{y \uparrow \beta} f_y(x)$ exist in $L^2(\R)$ and we have $f_0 \in \cD(e^{\beta P})$ and $e^{\beta P}f_0=f_\beta$. 

$(ii)$ For each function $f_0 \in \cD(e^{\beta P})$ there exists a unique function $f$ as in (i) such that $f_0:=\lim_{y \downarrow 0} f_y(x)$ in $L^2(\R)$  and $e^{\beta P}f_0=f_\beta$. 
\end{thl}
\begin{proof}
\cite[Lemma 1.1]{S1994}.
\end{proof}
If $f$ is a function as in Lemma \ref{ebatepcondition}(i), we  write simply $f(x)$ for $f_0(x)$ and $f(x+\ii \beta)$ for $f_\beta (x)$. Then the  operator $e^{\beta P}$ acts by 
\begin{align}\label{actionb}
(e^{\beta P})(x)=f(x+ \ii \beta),~~f \in  \cD(e^{\beta P}).
\end{align}

For a nonzero complex number $\sf{q}$ we denote by $\cS({\sf{q}})^+$  the closed sector in the  plane with opening angle less than $\pi$ between the positive $x$-axis and the half-line through  the origin and $\overline{\sf{q}}$ and  set $\cS({\sf{q}}):= \cS({\sf{q}})^+\cup(-\cS({\sf{q}})^+).$

\begin{figure}[h]
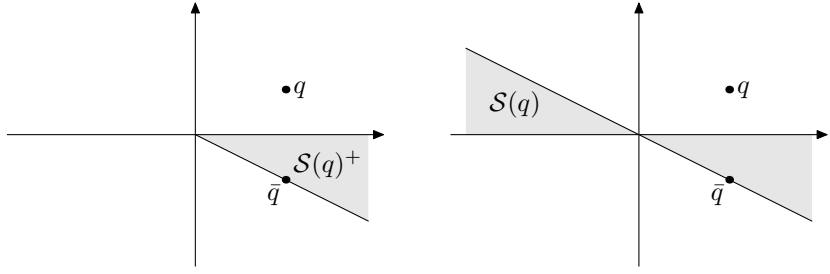

\unitlength=0.2ex
\hfil 
\includegraphics{001.mps} \qquad\includegraphics{002.mps}
\hfil
\caption{The sectors $\mathcal S(q)^+$ and $\mathcal S(q)$}
\end{figure}

We fix two reals $\alpha, \beta$ such that  $\beta>0$ and $0<|\alpha \beta|<\pi$.   Put ${\sf{q}}:=e^{-\ii \alpha \beta}$. 
Now we define positive selfadjoint operators $A$ and $B$ on the Hilbert space $L^2(\R)$ by
\begin{align}
{A}:=e^{\alpha Q}~~~~{\rm and}~~~~{B}:=e^{\beta P}.
\end{align}
\begin{thc}\label{sfabcomm}
If $f\in \cD(BA)\cap \cD(B)$, then $f\in \cD(AB)$ and $ABf={\sf{q}} BAf$.
\end{thc}
\begin{proof} Using the description of the domain $\cD(B)=\cD(e^{\beta P})$ in Lemma \ref{ebatepcondition} and formula (\ref{actionb}) we derive
\[
({\sf{q}}BAf)(x)= q B(e^{\alpha x}f(x))= e^{-\alpha \ii \beta } e^{\alpha(x+\ii \beta )} f(x+\ii\beta)=  e^{\alpha x} f(x+\ii\beta)=(ABf)(x).
\qedhere
\]
\end{proof}
Clearly,  the linear space 
$$ \cD_0={\rm Lin}~\{e^{-\varepsilon x^2+\gamma x}; \varepsilon >0, \gamma\in \C\}
 $$ 
 is contained in $\cD(A)\cap \cD(B)$ and 
 it is invariant under $A$ and $B$ and also under the Fourier transform and its inverse. By Corollary \ref{sfabcomm}, we have  $\cD_0\subseteq \cD_{\sf{q}}(A,B)$. As noted in \cite{S1994}, $\cD_0$ is a core for both selfadjoint operators $A$ and $B$. 

\begin{thp}\label{weaksanserifrel}
Suppose that $\lambda \in \C\backslash \cS({\sf{q}})^+$. Then\, $\lambda$ and $ \lambda {\sf{q}}$ are in $\rho(A)$ and  $${\sf{q}} R_{\lambda 
{\sf{q}}}(A)B\subseteq BR_\lambda(A).$$
\end{thp}
\begin{proof}
If $z$ runs through the strip $\{z: 0\le {\rm Im}~ z\le \beta\}$,  
then the number $e^{\alpha z}$ fills the sector $\cS({\sf{q}})^+$. 
Hence the infimum of 
the function\, 
$|e^{\alpha z}-\lambda |$ on the strip $\cI_\beta $ is equal to the distance of $\lambda$ from $\cS({\sf{q}})^+$. In particular, this infimum is  positive, since $\lambda \notin \cS({\sf{q}})^+$.

Let $f\in \cD(B)$ and let $f(z)$ be the corresponding holomorphic function from Lemma \ref{ebatepcondition}. Since $|e^{\alpha z}-\lambda |$ has a  positive infimum  on the strip $\cI_\beta $,   the function  $g(z)=(e^{\alpha z}-\lambda )^{-1}f(z)$ is  holomorphic on  $\cI_\beta$ and it satisfies condition  (\ref{ebatepcondition1}) as well, because $f$ does. Therefore, from Lemma \ref{ebatepcondition} we conclude that $g\in \cD(B)$ and 
\begin{align*} ({B} R_{\lambda }({A})f)(x) &= ({B} g)(x)=g(x+\ii \beta )=(e^{\alpha (x+\ii \beta )}-\lambda )^{-1}f(x+\ii \beta)\\&= 
{\sf{q}}(e^{\alpha x}-\lambda{ \sf{q}})^{-1}f(x+\ii \beta)=({\sf{q}} R_{\lambda \sf{q}} ({A}){B}f)(x).
\qedhere
\end{align*}
  \end{proof}

Another technical ingredient used below is Balakrishnan's theory of fractional powers of nonnegative operators on Banach spaces \cite{bal}, see e.g. \cite{CS}. 
 
 Suppose that $T$ is  a closed linear operator on a Banach space such that 
\begin{align}\label{asslemmafracpre}
(-\infty,0)\subseteq \rho(T)\quad {\rm and}\quad {\rm sup}~\{||\lambda (T+\lambda I)^{-1}||:  \lambda >0 \}<\infty.
\end{align}
Then, for any  $\gamma\in \C$,  $0<{\rm Re}\, \gamma<1$, the Balakrishnan operator $J^\gamma$ (see \cite{CS}, p. 57) is defined by  
\begin{align}\label{imaginarypowerforpre}
J^{\gamma}f =\frac{\sin (\varepsilon+\ii t)\pi}{\pi } \int_0^\infty \lambda^{\gamma- 1}
(T+\lambda I)^{-1}Tf~d\lambda,~~~\quad f\in \cD(J^\gamma):=\cD(T).
\end{align}
Here the integral is meant as an improper Riemann integral of a continuous function on $(0,+\infty)$ with values in the underlying Banach space.
The operator $J^\gamma$ (or its closure) is considered as a  power of the operator $T$ with exponent $\gamma$. 

For our investigations the following  special case is sufficient.
 \begin{thp}\label{imaginarypower}
Suppose that $A$ is a positive self-adjoint  operator on a Hilbert space $\Hh$ such that ${\rm ker}\, A=\{0\}$ and let $\vartheta\in \R$, $|\vartheta|<\pi$. Let $T$ denote the normal operator $e^{\ii \vartheta} A$ in $\Hh$. Then, for any $0<\varepsilon<1$, $t\in \R$ and $f\in \cD(T)=\cD(A)$ we  have
\begin{align}\label{imaginarypowerfor}
T^{\varepsilon+ \ii t}f =e^{\ii \vartheta \varepsilon} e^{-\vartheta t}A^{\varepsilon+ \ii t}f =\frac{\sin (\varepsilon+\ii t)\pi}{\pi } \int_0^\infty \lambda^{\varepsilon+\ii t- 1}
(T+\lambda I)^{-1}Tf~d\lambda , 
\end{align}
where the operators $T^{\varepsilon+ \ii t}$ and $A^{\varepsilon+ \ii t}$ are defined by the spectral functional calculus.
\end{thp}
\begin{proof}
Using that $|\vartheta|<\pi$ and $A\geq 0$  it is easily verified that the operator $T$ satisfies the conditions stated in (\ref{asslemmafracpre}). Hence formula (\ref{imaginarypowerforpre}) for the Balakrishnan operator $J^{\varepsilon +\ii t}$ holds. For the normal operator $T$  the closure of the operator $J^{\varepsilon +\ii t}$ is just the power $T^{\varepsilon+ \ii t}$ defined by the  functional calculus  (see Example 3.3.2 in \cite{CS}), where the principal branch of the complex power has to be taken. Further,  since $|\vartheta|<\pi$, we have $ T^{\varepsilon+ \ii t} =e^{\ii \vartheta \varepsilon} e^{-\vartheta t}A^{\varepsilon+ \ii t}$. Hence formula (\ref{imaginarypowerfor}) follows from (\ref{imaginarypowerforpre}).
\end{proof}

\begin{thl}\label{limitvaresplion}
If $A$ is a positive self-adjoint operator with trivial kernel, then 
$$\lim_{\varepsilon \to +0} A^\varepsilon f=f \quad {\rm for}\quad f\in \cD(A).
$$
\end{thl}
\begin{proof}
By the spectral calculus  we have 
$$\|A^{\varepsilon}f-f\|^2 =\int_0^\infty |\lambda^\varepsilon- 1|^2~ d\langle E(\lambda)f,f\rangle. 
$$
Passing to the limit $\varepsilon \to +0$ and using Lebesgue's dominated convergence theorem (by the assumption $f\in \cD(A)$)  we obtain the assertion.
\end{proof} 

\section{Two Classes of Well-behaved Representations of Relation (\ref{qplane1})}\label{twoclasses}

In this section we describe some  well-behaved   representations of relation (\ref{qplane1}). For this  we also restate some  results from \cite{S1994}.

Recall that $q=e^{-\ii \theta_0}$ and $0<|\theta_0|<\pi$ by (\ref{qtheta0}). 
Set $\theta_{1} :=\theta_0 -\pi$ if $\theta_0 >0$,  $\theta_{1} :=\theta_0 +\pi$ if $\theta_0 < 0$. Then we also have
$$
-q=e^{- \ii \theta_1} \quad {\rm and}\quad 0<|\theta_1|<\pi .
$$

If $A=0$ or if $B=0$, then $\cD_q(A,B)=\cD(B)$ resp. $\cD_q(A,B)=\cD(A)$ and it is obvious that the pair $\{A,B\}$  satisfies the relation (\ref{qplane1})  and  the  resolvent relations (\ref{resolventw}) and (\ref{resolvents}). We  call  pairs of the form $\{0,B\}$ and $\{A,0\}$ {\it trivial representations} of relation (\ref{qplane1}). 

Interesting representations of relation (\ref{qplane1}) are  the classes $\cC_0$ and $\cC_{1}$  defined as follows. 
  \begin{thd}\label{classc1} 
Suppose that\, ${\rm ker}\, A= {\rm ker}\, B=\{0\}$.   We say that the pair $\{A,B\}$  is  a {\rm  representation of the class} $\cC_0$ if 
  \begin{align}
  |A|^{\ii t} B &\subseteq e^{\theta_0 t}B |A|^{\-\ii t}~~, ~~ t\in \R,~~~   ~{\rm and}~~~~ U_A B \subseteq BU_A.
  \end{align}
 and that the pair $\{A,B\}$  is  a {\rm  representation of the class} $\cC_1$ if 

  \begin{align}
  |A|^{\ii t} B &\subseteq e^{\theta_1 t}B |A|^{\-\ii t}~~ , ~~ t\in \R,~~~~   ~{\rm and}~~~~ U_A U_B = - U_BU_A.
  \end{align}  
\end{thd}  

 \begin{thd}
 The trivial pairs $\{A_2,0\}$,  $\{0,B_2\}$ and pairs $\{A_0,B_0\}$  and  $\{A_{1},B_{1}\}$ of the classes $\cC_0$ and $\cC_{1}$, respectively, and  orthogonal direct sum of such pairs are called {\rm well-behaved representations} of relation (\ref{qplane1}).  

\end{thd} 
{\bf Remarks.} 1. Note that the class $\cC_0$ defined above is precisely the class $\cC_0$ in \cite{S1994}, while the class $\cC_1$ according to Definition \ref{classc1} corresponds to  $\cC_1$  if $\theta_0<0 $ and to  $\cC_{-1}$ if $\theta_0 >0$  in \cite{S1994}.

2. Suppose that $\{A,B\}$ is a well-behaved representation of relation (\ref{qplane1}). If $A\geq 0$ and ${\rm ker}\, A= {\rm ker}\, B=\{0\}$, then $U_A=I$ and  $\{A,B\}$ is a pair of the class $\cC_0$. Further, if $A\geq 0$, then the well-behaved representation $\{A,B\}$ cannot have an orthogonal summand of the class $\cC_1$. 

3. As it is usual for   relations having unbounded operator representations there are  many "bad" unbounded representations of relation (\ref{qplane1}). In \cite{S1992} pairs  of self-adjoint operators $A$ and $B$ have been constructed for which $\cD_q(A,B)$ is a core for $A$ and $B$, but the pair $\{A,B\}$ is not a well-behaved representation of  relation (\ref{qplane1}) and it is not in one of classes\, $\cC_n$, $n\in \Z$, defined in \cite{S1994}.
\medskip

Let us describe all pairs of the classes $\cC_0$ and $\cC_{1}$ up to unitary equivalence. We fix real numbers $\alpha,\alpha_1, \beta, \beta_1$, where $\beta>0$, $\beta_1>0$, such that
 \begin{align}
 \alpha \beta=\theta_0~\quad {\rm and}~\quad \alpha_1 \beta_1=\theta_1,\quad{\rm where}\quad q=e^{-\ii \theta_0}\quad~~{\rm and}\quad {-}q=e^{-\ii \theta_1}.
 \end{align}
Let $\cK$ be a Hilbert space.

Let $u,v$ be two commuting self-adjoint unitaries on $\cK$. We define self-adjont operators $A$ and $B$ on the Hilbert space $\Hh=\cK\otimes L^2(\R)$ a by
\begin{align}\label{unitarycc1}
 A_0= u\otimes e^{\alpha Q},\quad B_0=v\otimes e^{\beta P}
\end{align}
 and self-adjoint operators $A_1$ and $B_1$ on the Hilbert space $\Hh_1=(\cK\oplus\cK)\otimes L^2(\R)$ by the operator matrices
\begin{align}\label{unitarycc2}
   A_{1} =\begin{pmatrix} e^{\alpha_{1} Q}& 0\\&-e^{\alpha_{1} Q}\end{pmatrix},\quad  B_1= \begin{pmatrix}0&e^{\beta_1 P} \\ e^{\beta_1 P}&0
   \end{pmatrix}.
\end{align}
\begin{thp}\label{classpropwell}
The pairs $\{A_0,B_0\}$ and $\{A_1,B_1\}$ belong to the classes $\cC_0$ and $\cC_1$, respectively. Each pair of the class $\cC_0$ resp. $\cC_1$
is unitarily equivalent to a pair $\{A_0,B_0\}$ resp. $\{A_1,B_1\}$ of the form (\ref{unitarycc1})  resp. (\ref{unitarycc2}).
\end{thp}
\begin{thc}\label{wellbehcor1}
Up to unitary equivalence there are precisely five nontrivial irreducible well-behaved representations of relation (\ref{qplane1}). These are the fours pairs  $\{A= \varepsilon_1 e^{\alpha Q}, B=\varepsilon_2 e^{\beta P}\}$ on $L^2(\R)$, where $\varepsilon_1, \varepsilon_2\in \{+1,-1\}$, and the  pair $\{A,B\}$ on $\C^2\otimes L^2(\R)$ given by (\ref{unitarycc2}) with $\cK=\C$. 

Any well-behaved representation\, $\{A,B\}$\, of relation (\ref{qplane1}) satisfying\, ${\rm ker}\, A= {\rm ker}\, B=\{0\}$ is a direct orthogonal  sum of these representations.
\end{thc}
\begin{thc}\label{wellbehcor2}
Let $\{A,B\}$ be a well-behaved representation  of relation (\ref{qplane1}) for which\, ${\rm ker}\, A= {\rm ker}\, B=\{0\}$. Then there is a linear subspace $\cD\subseteq \cD(A)\cap\cD(B)$ such that\\
(i)~~ $A\cD=\cD$, $B\cD=\cD$, and\, $ |A|^{\ii t}\cD=\cD$, $|B|^{\ii t}\cD=\cD$\, for $t\in \R$,\\
(ii)~ $\cD$ is a core for $A$ and $B$,\\
(iii) $ABf=BAf$ for $f\in \cD$.
\end{thc}
\begin{thc}\label{wellbehcor3}
A pair  $\{A,B\}$ is a well-behaved representation  (resp. of the class $\cC_0$ or $\cC_1$) of relation (\ref{qplane1}) if and only if $\{B,A\}$ is a well-behaved representation (resp. of the class $\cC_0$ or $\cC_1$) of relation (\ref{qplane2}).
\end{thc}
Proposition \ref{classpropwell} and Corollaries \ref{wellbehcor1}--\ref{wellbehcor3} are contained in \cite[Section 2]{S1994}. 

The next proposition is essentially used in the proofs of various  theorems in Section \ref{characterizations}.

\begin{thp}\label{wellbehaveuaub}
Let $k=0,1$.
Suppose that ${\rm ker}~A={\rm ker}~B=\{0\}$ and $\cD_q(A,B)$ is a core for $B$. If\,  
\begin{align}\label{modabrel}
|A|^{\ii t}B\subseteq e^{\theta_k t} B|A|^{\ii t}\quad {\rm for}\quad t\in \R,
\end{align} 
then $\{A,B\}$ is a pair of the  class $\cC_k$.
\end{thp}
\begin{proof}
Putting $q_k:=(-1)^k q$ we have $q_k=e^{\ii \theta_k}$. 
By (\ref{modabrel}),  Proposition 2.3 in \cite{S1994} applies to the pair $\{|A|,B\}$ and the  relation $|A|B=q_kB|A|$. Hence   there exists a linear subspace $\cD$ of $\cD_{q_k}(|A|,B)$ such that $\cD=|A|\cD$ is a core for $B$. Then $|A|B g=q_kB|A|g$ for $g\in \cD$ by the definition of $\cD_{q_k}(|A|,B)$. Since $A$ is self-adjoint and ${\rm ker}~A=\{0\}$, $U_A$ is self-adjoint unitary and $A=|A|U_A$. 

Let $f\in \cD_q(A,B)$ and $g\in \cD$. Using the preceding facts  we derive 
\begin{align*}
\langle U_{A}f,B|A|g\rangle&=\langle U_{A}f,\overline{q_k} |A|B g\rangle=\langle f,\overline{q_k} U_A|A|B g\rangle=\langle f,\overline{q_k} A B g\rangle=\langle q_k A f,B g\rangle\\&=\langle(-1)^k  qB A f,g\rangle=\langle (-1)^k A Bf,g\rangle=\langle (-1)^k U_{A}Bf,|A|g\rangle .
\end{align*}
Since $\cD=|A|\cD$ is a core for $B$,  from  the preceding equality we conclude that $  U_{A}f\in \cD(B)$ and $BU_{A}f=(-1)^k U_{A}B f$ for $f\in \cD_q(A,B)$. By assumption $\cD_q(A,B)$ is a core for $B$, so the  latter implies that $U_AB \subseteq (-1)^k B U_A$. 

Since $U_A$ is a self-adjoint unitary, we get $U_AB U_A\subseteq (-1)^k B$, that is, the self-adjoint operator $(-1)^kB$ is an extension of the self-adjoint operator $U_ABU_A$ on $\Hh$. This is only possible if $U_AB U_A= (-1)^k B$. From the latter it follows that  $U_A|B| U_A=|B| $ and hence ~ $|B|U_A= U_AU_A|B| U_A=|B|U_A$. Therefore, 
$$
U_AU_B |B| =U_AB \subseteq (-1)^k B U_A= (-1)^k U_B|B|U_A=(-1)^k  U_BU_A|B |.
$$
Since $\ker B=\{0\}$, the range of $ |B|$ is dense in $\Hh$, so  we get $U_AU_B=(-1)^kU_AU_B$. Thus, $\{A,B\}$ is in  $ \cC_k$.
\end{proof}
 
Now let us return to the weak resolvent equations
\begin{align}\label{weakabid1}
&q R_{\lambda q}(A)B\subseteq BR_\lambda(A),\\& 
\overline{q} R_{\mu \overline{q}}(B)A\subseteq AR_\mu(B).\label{weakabid2}
\end{align}
For the trivial representations $\{0,B\}$ resp. $\{A,0\}$ they are obviously fulfilled for all $\lambda\neq 0$ resp. $\mu\neq 0.$ The classes $\cC_0$ and $\cC_1$ are treated in the next theorem.
\begin{tht}\label{weakclasscandc1}
(i)
If  $\{A,B\}$ is a pair of the class $\cC_0$ for the relation (\ref{qplane1}), then (\ref{weakabid1}) and (\ref{weakabid2}) are satisfied for all $\lambda\in\C\backslash \cS(q)$ and $\mu\in\C\backslash \cS({\overline{q}})$. If in addition $A\geq 0$ resp. $B\geq 0$,  then (\ref{weakabid1}) resp. and (\ref{weakabid2}) holds  for $\lambda\in\C\backslash \cS(q)^+$ resp. $\mu\in\C\backslash \cS({\overline{q}})^+$.\\
 (ii) If $\{A,B\}$ is a pair of the class $\cC_1$ for the relation (\ref{qplane1}), then (\ref{weakabid1}) and (\ref{weakabid2}) are fulfilled for  $\lambda\in\C\backslash \cS(-q)$ and $\mu\in\C\backslash \cS(-{\overline{q}})$. 
\end{tht}
\begin{proof}
By Corollary \ref{wellbehcor3} it suffices  to prove all assertions for the first relation (\ref{weakabid1}).  Clearly, (\ref{weakabid1}) is preserved under orthogonal direct sums. Therefore, by Corollary \ref{wellbehcor1}, it is sufficient to prove (\ref{weakabid1}) for the corresponding irreducible  representations listed in Corollary \ref{wellbehcor1}. 

(i): Let $\{A_0=  e^{\alpha Q}, B_0= e^{\beta P}\}$  and suppose that $\lambda\notin \cS(q)^+.$
Then, by Proposition \ref{weaksanserifrel}, relation (\ref{weakabid1}) is valid. Then, obviously,  (\ref{weakabid1}) holds  also for the  pair 
$\{A_0,- B_0\}$. Since $R_z(-A_0)=-R_{-z}(A_0)$, it follows  that (\ref{weakabid1}) is satisfied for the pairs $\{- A_0,\pm B_0\}$ provided that $\lambda \notin  -\cS(q)^+.$ 

(ii): We have to show that (\ref{weakabid1}) holds for the pair $\{A_1,B_1\}$ given by (\ref{unitarycc2}). Put ${A}:=e^{\alpha_1 Q}$, ${B}:=e^{\alpha_1 P}$ and ${\sf{q}}:=-q$. Then, since
\begin{equation}
 R_z( A) =\begin{pmatrix} R_z({A})&0\\0&R_z({\sf{-A}})\end{pmatrix},\quad  B = \begin{pmatrix}0&{B}\\{B}&0\end{pmatrix},
\end{equation} 
 the weak resolvent relation $q R_{\lambda q}( A) B\subseteq  B R_\lambda( A) $ reduces to the relations
\[
q R_{\lambda q}({A}) {B}\subseteq {B}R_\lambda(-{A})  , \quad q R_{\lambda q}(-{A}) {B}\subseteq {B}R_\lambda({A}).
\]
Since $R_z(-{A}) = -R_{-z}({A})$ and ${\sf{q}}:=-q$, the latter equalities are equivalent to
\[
{\sf{q}}  R_{(-\lambda) {\sf{q}}}({A}) {B}\subseteq {B}R_{(-\lambda)}({A})  , \quad {\sf{q}} R_{\lambda {\sf{q}}}({A}){B} \subseteq {B}R_\lambda({A}).  
\]
But these relations follow  from Proposition \ref{weaksanserifrel},  now applied to ${\sf{q}}:=-q$.
  \end{proof}

\section{Characterizations of Classes $\cC_0$ and $\cC_{1}$ by  Weak Resolvent Identities}\label{characterizations}

Recall that $A$ and $B$ always denote self-adjoint operators acting on a Hilbert space $\Hh$ and that the linear subspace $\cD_q(A,B)$ was defined in Definition \ref{defsete1}.
 
Let $\{A,B\}$ be a pair of class $\cC_0$. If $\lambda >0$, then $-\lambda$ is not in the sector $\cS(q)^+$, so the weak resolvent identity~  
$q R_{-\lambda q}(A)B\subseteq BR_{-\lambda}(A)$~ holds by Theorem \ref{weakclasscandc1}(i). Similarly, if $\mu\in \R \ii$, $\mu\neq 0$, and $|\theta|<\frac{\pi}{2}$, then $\mu\notin\cS(q)^+$ and hence~ $q R_{\mu q}(A)B\subseteq BR_{\mu}(A)$. The following two theorems state some converses of these assertions.

\begin{tht}\label{charpositive}
Let $A$ is a positive operator such that ${\rm ker}~A=\{0\}$. Suppose that 
$0<|\theta_0| <\pi$  and  the domain $\cD_q(A,B)$ 
is  a core for  $B$.  Assume that
\begin{align*}
q R_{-\lambda q}(A)B\subseteq  B R_{-\lambda}(A)~~~{\rm for}~~~~\lambda > 0.
\end{align*}
Then the pair  $\{A,B\}$ is an orthogonal  direct sum 
of a trivial representation $\{A_2,0\}$ and a pair $\{A_0,B_0\}$ of the class $\cC_0$.
 \end{tht}
\begin{proof}
Let $f\in \cD_q(A,B).$ 
Cleary, the positive self-adjoint operator $A$ and the  normal operator  $\overline{q}A$ (because of $q=e^{-\ii \theta_0}$ with $|\theta_0|<\pi$) satisfy the assumptions of Proposition \ref{imaginarypower}.  By the definition of $ \cD_q(A,B)$, the vectors $f$ and  $Bf$ are in $\cD(A)$, so formula (\ref{imaginarypowerfor}) applies to the operator $T=A$ and the vector $Bf$  and also to the operator $T=\overline{q}A$  and the vector $f$. The assumptions  $q R_{-\lambda q}(A)B\subseteq  B R_{-\lambda }(A)$ and $ABf=qBAf$ imply that $$(\overline{q}A+\lambda  I)^{-1}(\overline{q}A)Bf=B(A+\lambda I)^{-1}Af.$$
Next we apply Proposition \ref{imaginarypower} to the  operator $T=\overline{q}A$. Since $\overline{q}=e^{\ii \theta_0}$ with $|\theta_0|<\pi$, the assumptions of Proposition \ref{imaginarypower} are fulfilled.
Interchanging the closed operator $B$ and the integral in formula (\ref{imaginarypowerfor})  (by considering the integral as a limit of $\Hh$-valued Riemann sums)  we therefore obtain
\begin{align}\label{a1+it}
(\overline{q}A)^{\varepsilon+\ii t} B f= B A^{\varepsilon+\ii t}f~~\quad~~~~{\rm for}~~~f\in \cD_q(A,B),~~t\in \R.
\end{align}
By the first equality in (\ref{imaginarypowerfor}) and the relation $A^{\varepsilon+\ii t} f= A^\varepsilon  A^{\ii t} f$ we have
\begin{align*}
(\overline{q}A)^{\varepsilon+\ii t}Bf= (e^{\ii  \theta_0} A)^{\varepsilon+\ii t} Bf = e^{\ii \varepsilon \theta_0} e^{-\theta_0 t} \, A^{\ii t} A^\varepsilon Bf,
\end{align*} 
so  by (\ref{a1+it}) we obtain 
\begin{align}\label{rrlaion}  
 e^{\ii \varepsilon \theta_0} e^{-\theta_0 t} \, A^{\ii t} A^\varepsilon B f=B A^\varepsilon A^{\ii t}f.
\end{align} 
Recall that $f\in \cD(A)$ and $A^{\ii t} Bf\in \cD(A)$ by the assumption $f \in \cD_q(A,B)$. Passing to the limit $\varepsilon \to +0$ in (\ref{rrlaion}) by using Lemma \ref{limitvaresplion} and the fact that the operator $B$ is closed it follows that $ A^{\ii t}B f=  e^{\theta_0 t} B\, A^{\ii t} f$ for all $f\in \cD_q(A,B)$.   Since $\cD_q(A,B)$ is a core for $B$ by assumption, we conclude that $A^{\ii t}B \subseteq e^{\theta t} BA^{\ii t}$ for all $t\in \R$. The latter implies that $A^{\ii t}$ leaves the closed subspace $\Hh_2:={\rm ker}~B$ invariant. Hence $\Hh_2$ is reducing for $A$ and $B$, so we have $B =0\oplus B_0$ and $A=A_2\oplus A_0$ on $\Hh=\Hh_2 \oplus \Hh_2^\perp$ such that ${\rm ker}~ B_0=\{0\}$ and\, $(A_0)^{\ii t}B_0 \subseteq e^{\theta_0 t} B_0(A_0)^{\ii t}$\, for  $t\in \R$. Since $A\geq 0$ and ${\rm ker}~A=\{0\}$,  the pair  $\{A_0,B_0\}$ belongs to $\cC_0$.
\end{proof}

\begin{tht}\label{chargenerala}
Suppose that\, $0<|\theta_0| <\frac{\pi}{2}$ and ${\rm ker}~A=\{0\}$. Assume that 
$\cD_{q^2}(A^2,B)$  is a core for  $B$ and 
\begin{align*} 
q R_{\mu \ii q}(A)B\subseteq  B R_{\mu\ii}(A) ~\quad~~~{\rm  for}~~~ \mu \in \R, \mu\neq 0.
\end{align*}
 Then   $\{A,B\}$ is an orthogonal sum of a trivial representation  $\{A_2,0\}$ and a pair $\{A_0,B_0\}$ such that $\{|A_0|,B_0\}$ belongs to the class $ \cC_0$.
If in addition $\cD_q(A,B)$ is a core for $B$, then $\{A_0,B_0\}$ is a pair of the class $ \cC_0 $.
 \end{tht}
 \begin{proof}
From the relation\, $q R_{\mu \ii q}(A)B\subseteq  B R_{\mu\ii}(A)$\, it follows that each resolvent $R_{\mu\ii}(A)$ and its adjoint $R_{-\mu\ii}(A)$ leaves the  closed linear subspace $\Hh_2:={\rm ker}~B$ invariant. This implies that $\Hh_2$ is a reducing subspace for $R_{\mu\ii}(A)$ and therefore for the operator $A$. Obviously, $\Hh_2$ reduces $B$. Hence the pair $\{A,B\}$ is an orthogonal sum of a trivial representation $\{A_2,0\}$ and a pair $\{A_0,B_0\}$ such that ${\rm ker}~ B_0=\{0\}$. For notational simplicity let us assume already that ${\rm ker}~ B=\{0\}$. Our aim is to prove that $\{A,B\}$ is in  $\cC_{0}$.

 First we recall a simple operator-theoretic fact: If $T$ is a closed  operator such that $\nu, -\nu \in \rho(T)$, then $\nu^2 \in \rho(T^2)$ and
\begin{align}\label{residentity}
 R_{\nu^2} (T^2) = R_{\nu} (T)  R_{-\nu}(T)= \frac1{2\nu} (R_{\nu} (T) - R_{-\nu}(T)).
\end{align}

Suppose now that $\lambda >0$. Putting  $\mu = \sqrt{\lambda}$, we have $(\ii\mu) ^2= -\lambda$.
Let $f\in \cD(B)$. Using the identity (\ref{residentity}) twice,  for  $\nu=\mu \ii q$ and for $\nu=\mu \ii $, and the assumptions $qR_{\pm\mu \ii q}(A)B\subseteq  B R_{\pm\mu\ii}(A)$, we obtain  
\begin{align}\label{reslab}
 q^2 R_{-\lambda q^2}(A^2)B f &= \frac {q^2} {2\mu \ii q} ~ (R_{\mu \ii q} (A) - R_{-\mu \ii q }(A))Bf\nonumber \\&= \frac {1} {2\mu\ii}~ B (R_{\mu \ii} (A) - R_{-\mu \ii}(A))f =  B R_{-\lambda }(A^2) f .
\end{align}
Thus, since\, $q^2=e^{-2\ii \theta_0}$,\, $|2 \theta_0|< \pi$\, and\, $\cD_{q^2}(A^2,B)$ is a core for  $B$  by assumption, the pair $\{A^2,B\}$ satisfies all assumptions of Theorem \ref{charpositive} for the relation $A^2B=q^2BA^2$. Therefore, by this theorem we have~ $(A^2)^{\ii t}B \subseteq e^{2 \theta_0 t} B(A^2)^{\ii t}$ for $t\in \R$, so that\, $|A|^{\ii s}B \subseteq e^{ \theta_0 s} B|A|^{\ii s}$ for $s\in \R$. 
Hence the pair $\{|A|,B\}$ belongs to  $\cC_0$ (see e.g. Remark 2 in Section \ref{twoclasses}). 
Further, if in addition $\cD_{q}(A,B)$ is a core for  $B$,   
it follows  from Proposition \ref{wellbehaveuaub} that $\{A,B\}$ is in the class $ \cC_0 $.
\end{proof}

The next theorem contains a characterization of the class $\cC_{1}$. Recall that  for the class $\cC_1$ the weak resolvent relation (\ref{weakabid1}) holds  for $\lambda \in \C \backslash \cS(-q)$   by Theorem \ref{weakclasscandc1}(ii).

\begin{tht}\label{charclassc1sminusq} 
Suppose that $0<|\theta_0|< \pi$ and ${\rm ker}\, A={\rm ker}\, B=\{0\}$. Assume that both domains
$\cD_{q^2}(A^2,B)$ and $\cD_q(A,B)$ are cores for the operator  $B$ and that there exists a number $p\in \C\backslash \cS(-q)$ such that
\begin{align*}
q R_{\mu p q }(A)B\subseteq  B R_{\mu p}(A)\quad {\rm for ~all}\quad  \mu \in \R , \mu \neq 0. 
\end{align*}  
Then the pair $\{A,B\}$ belongs to the class $\cC_{1}$.
\end{tht}
\begin{proof}
Without loss of generality we can choose the number $p\in \C\backslash \cS(-q)$ of modulus one  and contained in the open sector with  angle less than $\pi$ between the positive $x$-axis and the half-line through the origin and $\overline{q}$.  
We  modify some arguments that have been used already in the proofs of Theorems \ref{charpositive} and \ref{chargenerala}.
\begin{figure}[h]
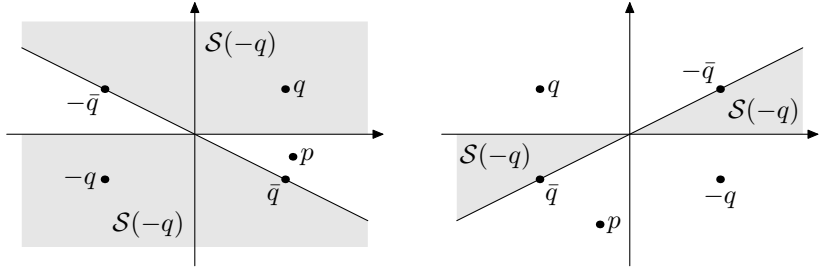

\hfil \includegraphics{003.mps}\qquad
\includegraphics{004.mps}\hfil
\caption{The $\cS(-q)$ sector and admissible ranges for $p$ as ${0<\arg q <\frac{\pi}{2}}$ and as ${\frac{\pi}{2}<\arg q <\pi}$}
\end{figure}

Using the identity (\ref{residentity}) the assumption $q R_{\mu p q }(A)B\subseteq  B R_{\mu p}(A)$ implies that
\begin{align}\label{p2squareres1} q^2 R_{\lambda p^2 q^2}(A^2)B \subseteq  B R_{\lambda p^2 }(A^2),~\lambda >0.
\end{align}
Let $f\in \cD_{q^2}(A^2,B)$.  From (\ref{p2squareres1}) and the relation $A^2Bf=q^2BA^2f$ we derive
\begin{align}\label{p2squareres2} 
(-\overline{p}^2\overline{q}^2A^2+\lambda I)^{-1}(-\overline{p}^2\overline{q}^2A^2)Bf=B(-\overline{p}^2A^2+\lambda I)^{-1}(-\overline{p}^2A^2)f,~\lambda >0.
\end{align}

Our aim is to apply Lemma \ref{imaginarypower}, especially  formula (\ref{imaginarypowerfor})  therein, to the operators\, $-\overline{p}^2\overline{q}^2A^2$\, and\, $-\overline{p}^2A^2$. To fullfill the assumptions of Lemma \ref{imaginarypower} it is crucial  to write the number\, $-\overline{p}^2\overline{q}^2$\, and\, $-\overline{p}^2$\, (of modulus one) in the form $e^{\ii \vartheta}$ with $\vartheta \in (-\pi,\pi)$. Let us write $p$ as $p=e^{\ii \psi}$ with $0<|\psi|<\pi$ and let $s(\psi)$ denote the sign of $\psi$. Note that $q=e^{-\ii \theta}$ and  $|\theta|<\pi$. Also we recall that by definition we have $\theta_1=\theta-\pi$ if $\theta>0$ and $\theta_1=\theta+\pi $ if $\theta <0$.
We shall prove that
\begin{align}\label{pqbarsquare1}
-\overline{p}^2&=e^{\ii(-2\psi +s(\psi)\pi)}\quad {\rm and}\quad -2\psi +s(\psi)\pi\in (-\pi,\pi),\\ -\overline{p}^2\overline{q}^2&= e^{\ii(2\theta_1-2\psi +s(\psi)\pi)}\quad {\rm and}\quad  2\theta_1-2\psi +s(\psi)\pi\in (-\pi,\pi).\label{pqbarsquare2}
\end{align}

First suppose that $\psi >0$. Then $\pi >\theta>\psi>0$ by the  choice of the number $p$. Hence   $-2\psi +\pi \in (-\pi,\pi)$ and $-\overline{p}^2=-(e^{-\ii \psi})^2= e^{\ii(-2\psi +\pi)}$. Further, 
$$2\theta_1-2\psi +\pi=2(\theta-\pi) -2\psi+\pi=2(\theta- \psi) -\pi \in (-\pi,\pi)$$ and 
$-\overline{p}^2\overline{q}^2=-(e^{-\ii \psi})^2(e^{\ii \theta})^2=e^{\ii (2 \theta -2\psi -\pi)}=e^{\ii(2\theta_1-2\psi +\pi)}$. 

Next we treat the case $\psi <0$. Then $0 >\psi >\theta>-\pi$ by the definition  of $p$.
Therefore, $-2\psi -\pi \in (-\pi,\pi)$ and $-\overline{p}^2=-(e^{-\ii \psi})^2= e^{\ii(-2\psi -\pi)}$. Moreover,
$$2\theta_1-2\psi -\pi=2(\theta+\pi) -2\psi-\pi=2(\theta- \psi) +\pi \in (-\pi,\pi)$$ and 
$-\overline{p}^2\overline{q}^2=-(e^{-\ii \psi})^2(e^{\ii \theta})^2=e^{\ii (2 \theta -2\psi +\pi)}=e^{\ii(2\theta_1-2\psi -\pi)}.$
This proves (\ref{pqbarsquare1}) and (\ref{pqbarsquare2}) in both cases.

By (\ref{pqbarsquare1}) and (\ref{pqbarsquare2}) it follows from  the first equality of formula (\ref{imaginarypowerfor}) that
\begin{align}\label{poweraespit1}
(-\overline{p}^2A^2)^{\varepsilon +\ii t}g&=e^{\ii(-2\psi +s(\psi)\pi)\varepsilon}\, e^{-(-2\psi +s(\psi)\pi)t}\, (A^2)^{\varepsilon +\ii t}g,\\
(-\overline{p}^2\overline{q}^2A^2)^{\varepsilon+\ii t}g&=e^{\ii(2\theta_1-2\psi +s(\psi)\pi)\varepsilon}\, e^{-(2\theta_1-2\psi +s(\psi)\pi)t}\, (A^2)^{\varepsilon +\ii t}g \label{poweraespit2}
\end{align}
for any $g\in \cD(A^2)$. Since $f,Bf\in \cD(A^2)$, the second equality of formula (\ref{imaginarypowerfor}) yields
\begin{align}\label{poweraespit3}
(-\overline{p}^2\overline{q}^2A^2)^{\varepsilon+ \ii t}Bf= B(-\overline{p}^2 A^2)^{\varepsilon +\ii t}f.
\end{align}
Inserting (\ref{poweraespit1}) with $g=f$ and (\ref{poweraespit2}) with $g=Bf$ into (\ref{poweraespit3}) and passing to the limit $\varepsilon\to +0$ we obtain
$$
e^{-2\theta_1 t}(A^2)^{\ii t}Bf =B( A^2)^{\ii t} f,~ t\in \R,
$$
and hence~ $|A|^{\ii s}Bf =e^{\theta_1 s}B | A|^{\ii s} f$,\, $s\in \R$,\, for all $f\in \cD_{q^2}(A^2,B)$. Since\, $\cD_{q^2}(A^2,B)$  is  a core for  $B$ by assumption, the latter implies that $|A|^{\ii s}B\subseteq e^{\theta_1 s}B | A|^{\ii s}$ for $s\in \R$. Therefore, since we also assumed that $\cD_q(A,B)$ is  a core for   $B$, the assumptions of Proposition \ref{wellbehaveuaub} are fulfilled with $k=1$, so the pair $\{A,B\}$ belongs to the class $\cC_{1}$.
\end{proof}
Related characterizations of  the class $\cC_1$ can be also given by requiring the weak resolvent relation for $|A|$ rather than $A$. The following theorem is a sample of such a result.

\begin{tht}\label{samplec1} Suppose that\, $0<|\theta_0|< \pi$ and ${\rm ker}\, A= {\rm ker}\, B=\{0\}$. Suppose  that the domains $\cD_q(A,B)$ 
and $\cD_{-q}(|A|,B)$  are  cores for  $B$~  and~~
\begin{align*} 
-q R_{\lambda q} (|A|)B\subseteq B R_{-\lambda} (|A|)\quad {\rm for}\quad  \lambda >0.
\end{align*}
Then the pair $\{A,B\}$ is in $\cC_1$. 
\end{tht}
\begin{proof}
The  assumptions of Theorem \ref{samplec1}  imply that the pair $\{|A|,B\}$ satisfies the assumptions of Theorem \ref{charpositive} for the relation (\ref{qplane1}) with $q=e^{-\ii \theta_0}$ replaced by $-q=e^{-\ii \theta_1}$. Note that $0<|\theta_1|< \pi$, since we assumed that $0<|\theta_0|< \pi$. Therefore, by Theorem \ref{charpositive}, we have $|A|^{\ii t}B \subseteq e^{\theta_1 t}B | A|^{\ii t} $ for all $t\in \R$. Since $\cD_q(A,B)$ is  a core for   $B$, it follows from Proposition \ref{wellbehaveuaub} that the pair $\{A,B\}$ is in $\cC_{1}$.
\end{proof}

The crucial assumption in the preceding   Theorems \ref{charpositive}, \ref{chargenerala}, and \ref{charclassc1sminusq} is that the weak $A$-resolvent identity (\ref{resolventw1}) holds on some line that intersects the critical sector only at the origin. In addition there have been  technical assumptions such as ${\rm ker}~A=\{0\}$ and the requirement that  $\cD_q(A,B)$ resp. $\cD_{q^2}(A^2,B)$ is  a core  for the operator $B$. These technical assumptions can be avoided if we assume in addition that the  weak $B$-resolvent identity (\ref{resolventw1-b}) holds for some points of the resolvent set $\rho(B)$. 
\begin{tht}\label{charpositive1}
Let $A$ is a positive operator. Suppose that  
$0<|\theta_0| <\pi$ and  there exist a number\,   $\nu \in \rho(B)$ such that\, $\nu \overline{q},\, \overline{\nu}\, \overline{q} \in \rho(B)$,  
\begin{align}\label{bresolventid1}
\overline{q}\, R_{\nu \overline{q}}(B)A\subseteq  A R_{\nu }(B)\quad {\rm and}\quad \overline{q}\, R_{\overline{\nu}\, \overline{q} }(B)A\subseteq  A R_{\overline{\nu}}(B).
\end{align}  Assume that
\begin{align}\label{aresolventid1}
q R_{-\lambda q}(A)B\subseteq  B R_{-\lambda}(A)\quad {\rm for~ all}\quad \lambda > 0.
\end{align}
Then the pair  $\{A,B\}$ is an orthogonal  direct sum  
of  trivial representations $\{A_2,0\}$ and $\{0,B_2\}$ and a pair $\{A_0,B_0\}$ of the class $\cC_0$.
\begin{proof}
From\, (\ref{bresolventid1}) it follows that the operator $R_{\nu }(B)$ and its adjoint $R_{\overline{\nu}}(B)$ leave the closed subspace 
$\cG_2:={\rm ker}~A$\, invariant. Therefore,  $\cG_2$ is reducing for $R_\nu (B)$ and hence for $B$. Since $\cG_2\equiv {\rm ker}~A$\, is obviously reducing for $A$, the pair  $\{A,B\}$  decomposes as an orthogonal sum of pairs\, $\{0,B_2\}$\, and\, $\{\tilde{A},\tilde{B}\}$ such that\, ${\rm ker}~\tilde{A}=\{0\}$.\, Further, by Proposition \ref{rldb}(vi), (\ref{bresolventid1}) implies that $\cD_q(A,B)$ is core for $B$. Hence $\cD_q(\tilde{A},\tilde{B})$ is core for $\tilde{B}$. Clearly, (\ref{aresolventid1}) leads to the relation $ R_{-\lambda q}(\tilde{A})\tilde{B}\subseteq  \tilde{B} R_{-\lambda}(\tilde{A})$ for $ \lambda > 0$. Thus, the pair $\{\tilde{A},\tilde{B}\}$ satisfies the assumptions of Theorem \ref{charpositive} which gives the assertion.
\end{proof}

 \end{tht}\begin{tht}\label{chargenerala1}
Suppose that\, $0<|\theta_0| <\frac{\pi}{2}$ and there exists a number  $\nu \in \rho(B)$ such that\, $\nu \overline{q},\, \overline{\nu}\, \overline{ q},\, \nu\, \overline{q}^2\in \rho(B)$ and 
\begin{align}\label{bresolventid2}
\overline{q}\, R_{\lambda \overline{q}}(B)A\subseteq  A R_{\lambda}(B)\quad {\rm for}\quad \lambda=\nu \overline{q}\, ,~ \nu,~ \overline{\nu}.
\end{align}
Assume that 
\begin{align*} 
q R_{\mu \ii q}(A)B\subseteq  B R_{\mu\ii}(A) ~\quad~~~{\rm  for~all}\quad \mu \in \R, \mu\neq 0.
\end{align*}
 Then   $\{A,B\}$ is an orthogonal sum of  trivial representations   and a pair $\{A_0,B_0\}$ belonging to the class $ \cC_0$.
 \end{tht}
 \begin{proof}
Arguing in a similar manner as in the  proof of Theorem \ref{charpositive1} the assertion is reduced to    
 Theorem \ref{chargenerala}. We sketch only the necessary modifications:  
From the relations  (\ref{bresolventid2}), applied for $\lambda=\nu,\, \overline{\nu}$, it follows  that\, ${\rm ker}~A$\, is reducing for the pair   $\{A,B\}$. By Proposition \ref{rldb}(vi) and Corollary \ref{da2core}, the relations  (\ref{bresolventid2}), applied with $\lambda=\nu,\, \nu \overline{q}$,\, imply that the domains $D_q(A,B)$ and $D_{q^2}(A^2,B)$ are cores for $B$. 
\end{proof}
\begin{tht}\label{charclassc1sminusq1} Suppose that $\frac{\pi}{2}<|\theta_0|< \pi$ and  there exist  numbers $\nu\in \rho(B)$ and $p\in \C\backslash \cS(-q)$ such that\, $\nu \overline{q},\, \overline{\nu}\, \overline{ q},\, \nu\, \overline{q}^2\in \rho(B)$,
\begin{align}\label{bresolventid3}
\overline{q}\, R_{\lambda \overline{q}}(B)A&\subseteq  A R_{\lambda}(B)\quad {\rm for}\quad \lambda=\nu \overline{q}\, ,~ \nu,~ \overline{\nu} ,\\
q R_{ \overline{p}\, q }(A)B&\subseteq  B R_{\overline{p}}(A),\label{aresolventid11}\\
q R_{\mu p q }(A)B&\subseteq  B R_{\mu p}(A)\quad {\rm for ~all}\quad  \mu \in \R , \mu \neq 0. \label{aresolventid2}
\end{align}  
Then $\{A,B\}$ is an orthogonal sum of  trivial representations  and a pair $\{A_1,B_1\}$ of the class $\cC_{1}$.
\end{tht}
\begin{proof}
As in the proof of Theorem \ref{chargenerala1} it follows from  (\ref{bresolventid3})   that\, 
${\rm ker}~A$\, is reducing for the pair   $\{A,B\}$ and that\, $D_q(A,B)$ and $D_{q^2}(A^2,B)$ are cores for $B$. Likewise the relations\, 
$q R_{ \overline{p}\, q }(A)B\subseteq  B R_{\overline{p}}(A)$ and $q R_{ p q }(A)B\subseteq  B R_p(A)$ 
 (by (\ref{aresolventid11}) and (\ref{aresolventid2})) imply that\, ${\rm ker}~B$\, is  reducing for the pair   $\{A,B\}$. Using these facts the assertion is derived from  Theorem \ref{charclassc1sminusq}.\end{proof}
{\bf Remarks.}  By Theorem \ref{weakclasscandc1},  the weak $B$-resolvent relation\, 
$\overline{q} R_{\mu \overline{q}}(B)A\subseteq AR_\mu(B)$ 
is satisfied for a pair $\{A,B\}$ of the class $\cC_0$ if\, $\mu\in\C\backslash \cS({\overline{q}})$\, and for  a pair $\{A,B\}$ in $\cC_1$ if\, $\mu\in\C\backslash \cS(-{\overline{q}})$.  From this result it follows that for the corresponding pairs in Theorems\, \ref{charpositive1}--\ref{charclassc1sminusq1}\, the assumptions (\ref{bresolventid1}), (\ref{bresolventid2}), and (\ref{bresolventid3}) can be fulfilled. That is, if  $\{A,B\}$ is a pair of the class $\cC_0$ we can choose  $\nu \in \C\backslash \cS({\overline{q}})$ such that $\overline{\nu} \in \C\backslash \cS({\overline{q}})$\, (then (\ref{bresolventid1}) holds) and if in addition $0<|\theta_0| <\frac{\pi}{2}$\, there exists $\nu \in \C\backslash \cS({\overline{q}})$\, such that $\nu \overline{q},\, \overline{\nu}\in \C\backslash \cS({\overline{q}})$ (which implies  (\ref{bresolventid2})). If the pair $\{A,B\}$ is in $\cC_1$ and $\frac{\pi}{2}<|\theta_0|< \pi$ we can find $\nu \in \C\backslash \cS(-{\overline{q}})$  such that $ \nu \overline{q}\, 
, \overline{\nu}\in \C\backslash \cS(-{\overline{q}})$ (these conditions imply (\ref{bresolventid3})) and there exists $ p\in \C\backslash \cS(-q)$ such that\, $ \overline{p}\, \in   \C\backslash\cS(-q)$ (then  (\ref{aresolventid1}) and (\ref{aresolventid2}) are satisfied by Theorem \ref{weakclasscandc1}(ii)).

\section{Deficiency subspaces and their dimensions}\label{defectspaces}

Let $A$ and $B$ be positive self-adjoint  operators with trivial kernels acting on a Hilbert space $\Hh$. Then, by Theorems~\ref{weakclasscandc1} and \ref{charpositive} and by Proposition~\ref{rldb},
the pair $\{A,B\}$ belongs to the class $\cC_0$ of representations of equation \eqref{qplane1} if and only if the resolvent relations \eqref{resolvents1} and \eqref{resolvents1-b} are satisfied  for all $\lambda,\mu\in \C$ such that   $\lambda,\mu\notin\cS(q)^+$. Moreover,  up to unitary equivalence, the only  irreducible such pair  is  $\{e^{\alpha Q}, e^{\beta P}\}$, where $|\alpha\beta|<\pi$ and $q=e^{-\ii\alpha\beta}$; see  Section~\ref{operatorpreliminaries}. In this section we study 
 the resolvent relations for  the pair $\{e^{\alpha Q},e^{\beta P}\}$  in the general case, that is, we do  not assume that $|\alpha\beta|<\pi$ or $\lambda,\mu\notin\cS(q)^+$.
 
First let us fix some assumptions and notations that will be kept throughout this section. Suppose that $\alpha,\beta \in \R$ and $q:=e^{-\ii\alpha\beta}\neq \pm 1$.  
We consider the pair $$\{A:=e^{\alpha Q}, B:=e^{\beta P}\}$$ of  self-adjoint operators on the Hilbert space $\Hh:=L^2(\R)$. Recall that $A$ and $B$ act by
\begin{align}\label{abrel}
(Af)(x) =e^{\alpha x}f(x)\quad {\rm and}\quad (Bf)(x) = f(x+\ii \beta)
\end{align}
for all functions $f$ of the dense domain
 $$\cD_0={\rm Lin}~\{e^{-\varepsilon x^2+\gamma x}; \varepsilon >0, \gamma\in \C\}.$$
Since  \eqref{qplane1} is satisfied for $f\in\cD_0$ (by (\ref{abrel})), we know from Section \ref{prelimrel1} that
 \eqref{resolvents} holds for all $h\in(B-\mu I)(A-\lambda I)\cD_0$ and that \eqref{resolvents-b} holds for all 
$h\in (A-\lambda q I)(B-\mu q I)\cD_0$.  Therefore, in order to find the  resolvent equations in the present case it suffices to describe the form of resolvent  equations on the orthogonal complements of  spaces $(B-\mu I)(A-\lambda I)\cD_0$ and  $(A-\lambda q I)(B-\mu q I)\cD_0$, respectively. This will be achieved by the formulas at the end of this section.
\begin{thd}
\begin{align*}
\Hh_A(\lambda,\mu)&=\{\psi \in \Hh: \psi\perp (B-\mu I)(A-\lambda I)\cD_0\},
\\
\Hh_B(\lambda,\mu)&=\{\eta\in \Hh:\eta \perp (A-\lambda q I)(B-\mu q I)\cD_0\}.
\end{align*}
\end{thd}
Assume that $\lambda$, $\mu \in \mathbb C \setminus (\mathbb R_+ \cup \bar q\mathbb R_+)$, where $\R_+=[0,+\infty)$,  and write 
$$
\lambda=e^{r+\ii s},\ \mu= e^{u+\ii v},\quad r>0,\ u>0, \ |s|<\pi,\ |v|<\pi.
$$
Further, we let 
$$
\alpha\beta =\theta_0+2\pi m,\quad m \in \mathbb Z,\ \theta_0 \in (-\pi, \pi),
$$
 so that $q=e^{-\ii \theta_0}$, and set 
$$
\epsilon_1=\sign(v),\ \epsilon_2 = \sign(v-\theta_0),\ \epsilon_3 = \sign(s),\ \epsilon_4 = \sign(s-\theta_0).
$$

\begin{tht}\label{maindef}
${\rm (i)}$ The vector space $\Hh_A(\lambda,\mu)$ is spanned by the functions 
\[
\psi_j(x)= 
\frac{
\bar\mu^{-\ii x/\beta}
e^{\pi(1+\epsilon_1) x/\beta}
}{
e^{2\pi x/\beta}
-
\bar\lambda^{2\pi/\alpha\beta}
e^{-4\ii\pi^2 j/\alpha\beta}
}, 
\]
where $j\in \mathbb Z$ and\, $0<(s - 2\pi j) /\alpha\beta<1$, and its dimension is 
\[
 \dim \Hh_A(\lambda,\mu) = \sign(\alpha\beta)\Bigl  (m + \frac{\epsilon_3-\epsilon_4}{2}\Bigr) .  
\]

${\rm (ii)}$  The space $\Hh_B(\lambda,\mu)$ is the linear span of functions
\[
 \eta_k(x) = 
\frac{
(\bar\mu \bar q)^{-\ii  x/\beta}
e^{2\pi kx/\beta}
}{e^{\alpha x} -\bar \lambda \bar q}.
\]
where $k\in \mathbb Z$ and\, $0<(\theta_0 + 2\pi k-v)/\alpha\beta <1$, and it has the dimension
\[
 \dim \Hh_B(\lambda,\mu) = \sign(\alpha\beta) \Bigl(m + \frac{\epsilon_1-\epsilon_2}{2}\Bigr).
\]
\end{tht}

\begin{proof}
(i): First we study the space $\Hh_B(\lambda,\mu)$. Let $\eta\in\Hh_B(\lambda,\mu)$. Applying (\ref{abrel}) to $\phi(x)=e^{-\epsilon x^2 +itx} \in \cD_0$ we compute
\begin{align*}
 ((A-\lambda q)(B-\mu q)\phi )(x)& = (e^{\alpha x} - \lambda q) (e^{-\epsilon (x + i\beta)^2 + it(x+i\beta)} -\mu q e^{-\epsilon x^2 + itx})
\\
&=(e^{\alpha x} - \lambda q) (e^{-2i\beta\epsilon  x + \epsilon \beta^2 - t\beta} -\mu q)
 e^{-\epsilon x^2 + itx}.
\end{align*}
Therefore, since  $\eta \perp (A-\lambda q)(B-\mu q)\phi(x) $,  we obtain
\begin{align*}
 0&=\int_{\mathbb R} \overline{\eta(x)} \,(e^{\alpha x} - \lambda q) (e^{-2i\beta\epsilon  x + \epsilon \beta^2 - t\beta} - \mu q) e^{-\epsilon x^2 + itx} \, dx 
\\
& = e^{\epsilon \beta^2 - t\beta} \int_{\mathbb R}e^{i(t-2\beta\epsilon)x}\, \overline{\eta(x)} \,(e^{\alpha x} - \lambda q) \,  e^{-\epsilon x^2} \, dx 
\\
&\qquad - \mu q \int_{\mathbb R}e^{itx}\, \overline{\eta(x)} \,(e^{\alpha x} - \lambda q) \, e^{-\epsilon x^2} \, dx
\end{align*}
which implies that
\begin{equation}\label{gepsilon}
 e^{\epsilon \beta^2 - t\beta} g_\epsilon (t-2\beta\epsilon) = \mu q\, g_\epsilon(t), \qquad g_\epsilon(t) = \int_{\mathbb R}e^{itx}\, \overline{\eta(x)} \,(e^{\alpha x} - \lambda q) \, e^{-\epsilon x^2} \, dx.
\end{equation}
Since $\mu=e^{u+iv}$, $|v|<\pi$, it is  easily seen that the function 
\[
G_\epsilon(t) = e^{-\frac{t^2}{4\epsilon} -\frac{u+i(v-\theta_0)}{2\epsilon\beta}t }
\] 
satisfies \eqref{gepsilon}. Therefore, any solution of \eqref{gepsilon} has the form\, $G_\epsilon(t)H_\epsilon(t)$, where $H_\epsilon$ is a  periodic function on $\R$  with period $-2\epsilon\beta$, that is, $H_\epsilon(t-2\epsilon\beta) = H_\epsilon (t)$ for $t\in \R$.  

The crucial step of this proof is contained in the following lemma.
\begin{thl}\label{trigpol}
 $H_\epsilon$ is a trigonometric polynomial, that is, there are an integer $l\ge0$ and numbers $c_k\in\mathbb C$, $k=-l,\dots, l$ such that  $H_\epsilon(t) =\sum_{k=-l}^l d_k e^{\frac{\ii \pi k t}{\epsilon\beta}}$. 
\end{thl}

\begin{proof}
 We have
\begin{equation}\label{function_H}
 H_\epsilon(t) = (G_\epsilon(t))^{-1}g_\epsilon(t) = 
 e^{\frac{t^2}{4\epsilon} +\frac{u+i(v-\theta_0)}{2\epsilon\beta}t } 
 \int_{\mathbb R}e^{itx}\, \overline{\eta(x)} \,(e^{\alpha x} - \lambda q) \, e^{-\epsilon x^2} \, dx.
\end{equation}
Because the above arguments are valid  for  complex numbers $t$ as well,   $H_\epsilon$ becomes  a periodic function on a whole complex plane $\mathbb C$.

Since $\eta \in L_2(\mathbb R)$,  the function $e^{\tau |x|}\overline{\eta(x)} \,(e^{\alpha x} - \lambda q) \, e^{-\epsilon x^2}$ is also  in $L_2(\mathbb R)$ for any $\tau>0$. Therefore, the integral in \eqref{function_H} is an entire function on the complex plane $\mathbb C$.

We show that $H_\epsilon$ is of exponential type. For any $s\in \mathbb R$ we have
\begin{align*}
 |H_\epsilon(t+\ii s)| &=  
| e^{\frac{(t+\ii s)^2}{4\epsilon} +\frac{u+\ii(v-\theta_0)}{2\epsilon\beta}(t+\ii s) }| 
\Bigl | \int_{\mathbb R}e^{\ii (t+ \ii s) x}\, \overline{\eta(x)} \,(e^{\alpha x} - \lambda q) \, e^{-\epsilon x^2} \, dx\Bigr |
\\
& =  e^{\frac{t^2 - s^2}{4\epsilon} +\frac{ut-(v-\theta_0)s}{2\epsilon\beta}  } 
\Bigl | \int_{\mathbb R}e^{\ii t x}\, \overline{\eta(x)} \,(e^{\alpha x} - \lambda q) \, e^{-\epsilon x^2 -sx} \, dx\Bigr |
\\
&\le  e^{\frac{t^2 - s^2}{4\epsilon} +\frac{ut-(v-\theta_0)s}{2\epsilon\beta}  } 
 \|\eta\| \Bigl (\int_{\mathbb R} |e^{\alpha x} - \lambda q|^2 \, e^{-2\epsilon x^2 -2sx} \, dx\Bigr )^{1/2}.
\end{align*}

Since 
\[
 \int_{\mathbb R} |e^{\alpha x} - \lambda q|^2 \, e^{-2\epsilon x^2 -2sx} \, dx = \sqrt{\frac{\pi}{2\epsilon}} e^{\frac{s^2}{2\epsilon}} \bigl(|\lambda|^2 + e^{\frac{\alpha(\alpha - 2s)}{2\epsilon}} -(\lambda q +\bar\lambda \bar q)e^{\frac{\alpha(\alpha - 4s)}{8\epsilon}}\bigr),
\]
for\, $0\le t \le 2\epsilon\beta$\, we have
\[
 |H_\epsilon(t+\ii s)| \le  
 Me^{-\frac{(v-\theta_0)s}{2\epsilon\beta}  } 
 \|\eta\| \Bigl ( \sqrt{\frac{\pi}{2\epsilon}} \bigl(|\lambda|^2 + e^{\frac{\alpha(\alpha - 2s)}{2\epsilon}} -(\lambda q +\bar\lambda \bar q)e^{\frac{\alpha(\alpha - 4s)}{8\epsilon}}\bigr)\Bigr)^{1/2},
\]
where $M$ is the supremum of $e^{\frac{t^2}{4\epsilon} +\frac{ut}{2\epsilon\beta}}$ on the interval $[0, 2\epsilon\beta]$. Since $H_\epsilon(t+\ii s)$ is periodic in $t$, the latter estimate holds for all $t\in \mathbb R$. Hence $H_\epsilon(t+\ii s)$ has exponential growth. Thus, $H_\epsilon$ is an entire periodic function of exponential type. By a result from complex analysis (see, e.g., \cite[p.~334]{mar}), such a function $H_\epsilon$ has to be  a trigonometric polynomial
\[
 H_\epsilon(t) = \sum_{k=-l}^l c_k e^{\frac{\ii \pi kt}{\beta\epsilon}}. \qedhere 
\]
\end{proof}

By Lemma \ref{trigpol} the equality $g_\epsilon(t) = G_\epsilon(t)H_\epsilon(t)$ takes the form
\[
 \int_{\mathbb R}e^{itx}\, \overline{\eta(x)} \,(e^{\alpha x} - \lambda q) \, e^{-\epsilon x^2} \, dx = e^{-\frac{t^2}{4\epsilon} -\frac{u+i(v-\theta_0)}{2\epsilon\beta}t } \sum_{k=-l}^l c_k e^{\frac{\ii \pi kt}{\beta\epsilon}}. 
\]
Applying the Fourier transform, we  obtain
\begin{align*}
 \overline{\eta(x)} \,(e^{\alpha x} - \lambda q) \, e^{-\epsilon x^2} & = \frac{1}{2\pi} \int_{\mathbb R} e^{-\ii tx}  g_\epsilon(t)\, dt
\\
&= e^{-\epsilon x^2}e^{\ii x u/\beta} e^{-(v-\theta_0)x/\beta} \sum_{k=-l}^l d_k e^{2\pi k x/\beta},
\end{align*}
where
\[
 d_k = \frac{\sqrt{\epsilon}}{\sqrt{\pi}} c_k e^{(u+\ii(v-\theta_0-2k\pi))^2/(4\epsilon\beta^2)}.
\]
Obviously, 
the factor $e^{-\epsilon x^2 }$  cancels, so we get
\begin{equation}\label{function_eta}
 \overline{\eta(x)}  = \frac{e^{\ii x u/\beta} e^{-(v-\theta_0)x/\beta}}
  {e^{\alpha x} - \lambda q}  \sum_{k=-l}^l d_k e^{\frac{2\pi k x}{ \beta} } 
= \frac{(\mu q)^{\ii x /\beta} }
  {e^{\alpha x} - \lambda q}  \sum_{k=-l}^l d_k e^{{2\pi k x}/{ \beta} }.
\end{equation}
Let us introduce the functions 
\[
 \eta_k(x) := \frac{(\bar\mu \bar q)^{-\ii x /\beta} e^{{2\pi k x}/{ \beta} } }
  {e^{\alpha x} - \bar\lambda \bar q}  
= \frac{e^{-\ii u  x /\beta }
  e^{(\theta_0- v +2\pi k ) x /\beta}
}
  {e^{\alpha x} - \bar\lambda \bar q}\, .
\]
Then we have $\eta(x) = \sum_{k=-l}^l \bar d_k\eta_k(x)$.

Next we want to decide which functions $\eta_k$  belong to $L_2(\mathbb R)$. Let us begin with the case where $\alpha>0$. Then the function $\eta_k(x)$
behaves like $e^{(\theta_0- v +2\pi k )x/\beta}$ as $x\to-\infty$ and like $e^{(\theta_0- v +2\pi k -\alpha\beta)x/\beta}$ as $x\to +\infty$, so that $\eta_k\in L^2(\R)$ if and only  $0<(\theta_0- v +2\pi k )/\beta<\alpha$. In the case where $\alpha<0$ a similar reasoning shows  that 
$\eta_k \in L_2(\mathbb R)$
if and only $\alpha<(\theta_0- v +2\pi k )/\beta<0$. Thus, in both cases we have $\eta_k\in L^2(\R)$ if and only if  if and only if\, $0<(\theta_0- v +2\pi k )/\alpha\beta<1$, that is,
\begin{align*}
v-\theta_0<2\pi k <v+2\pi m &\quad{\rm for}\quad \alpha\beta>0, 
\\
v+2\pi m<2\pi k < v-\theta_0& \quad{\rm for}\quad \alpha\beta<0,
\end{align*}
or equivalently,
\[
 \frac{1+\epsilon_2}{2}\le k \le m - \frac{1-\epsilon_1}{2} \quad \mbox{for $\alpha\beta>0$},
\quad m+\frac{1+\epsilon_1}{2} \le k \le \frac{\epsilon_2-1}{2} \quad \mbox{for $\alpha\beta<0$}.
\]
The functions $\eta_k$ are obviously linearly independent. From the asymptotic behaviour of $\eta_k$ it follows that\, $\eta(x) = \sum_{k=-l}^l \bar d_k\eta_k(x)$\, is  in $L^2(\R)$ if and only each $\eta_k$ with nonvanishing coefficient $\bar d_k$ is in $L^2(\R)$. Therefore, we obtain
\[
 \dim \Hh_B = \sign(\alpha\beta) \Bigl(m + \frac{\epsilon_1-\epsilon_2}{2}\Bigr).
\]

(ii): Now we turn  to the space $\Hh_A(\lambda,\mu)$. Let $\psi\in L^2(\R)$. Recall that $\psi\in \Hh_A(\lambda,\mu)$ if and only if
\[
\psi\perp (B-\mu  I)(A-\lambda  I)\cD_0= (e^{\beta P} -\mu I)(e^{\alpha Q} - \lambda I)\cD_0 
\]
For the Fourier transform $\mathcal F$ we have $e^{\alpha Q} = \mathcal F e^{-\alpha P} \mathcal F^{-1}$ and $e^{\beta P} = \mathcal F e^{\beta Q} \mathcal F^{-1}$. Hence the latter is equivalent to
\[
 \psi\perp 
\mathcal F (e^{\beta Q} -\mu I)(e^{-\alpha P} - \lambda I) \mathcal F^* \cD_0.
\]
Since $\cD_0$ is invariant under the Fourier transform, $\psi\in \Hh_A(\lambda,\mu)$ if and only if  the inverse Fourier transform\, $\cF^*\psi$\, of $\psi$ satisfies 
\[
 \mathcal F^* \psi \perp (e^{\beta Q} -\mu I)(e^{-\alpha P} - \lambda I)\mathcal D_0.
\]
From the proof of (i) we already know that $\psi\in \Hh_A(\lambda,\mu)$ if and only if $\cF^*\psi$ is a linear combination of functions
\begin{equation*}
 \phi_k(x) = 
\frac{(\bar\lambda)^{\ii x/\alpha} e^{-2\pi k x/\alpha}}{e^{\beta x} -\bar\mu}\, , 
\end{equation*}
where $k\in \mathbb Z$ and $0<(s - 2\pi k)/\alpha\beta<1$.
The condition on $k$ can be rewritten as
\begin{align*}
 \frac{1+\epsilon_4}{2}-m \le k \le  \frac{\epsilon_3-1}{2}& \quad {\rm for}\quad \alpha\beta>0,
\\
 \frac{1+\epsilon_3}{2} \le k \le \frac{\epsilon_4-1}{2} -m& \quad {\rm for}\quad \alpha\beta<0.
\end{align*}
Therefore, we have
\[
 \dim \Hh_A = \sign(\alpha\beta)\Bigl  (m + \frac{\epsilon_3-\epsilon_4}{2}\Bigr) .  
\]

To calculate the Fourier transform of $\phi_k$, we shall apply the following formula
(see, e.g., \cite[3.311.9]{gradshtein})
\begin{equation}\label{integral1}
 \int_{\mathbb R} \frac{e^{-\delta x}}{e^{-x}+\gamma}\, dx = \frac{\pi \gamma^{\delta -1}}{\sin \pi\delta}, \qquad |\arg \gamma|<\pi, \ 0< \mathop{\mathrm {Re}} \delta <1.
\end{equation}
After some computations using \eqref{integral1} we obtain\,
$\mathcal F \phi_k (x)= C_k \psi_k(x),$ where
\begin{align}
C_k
&= \frac{
\ii
\sqrt{2\pi}
e^{\ii(u-\ii v-\ii\pi(1-\epsilon_1))( r -\ii s  +\ii 2\pi k)/\alpha\beta}
}{
\bar\mu
\beta
}\, ,
\notag
\\
\psi_k(x)
&= \frac{
\bar\mu^{-\ii  t/\beta}
e^{\pi(1+\epsilon_1)  t /\beta}
}{
(
e^{2\pi t /\beta}
-
\bar\lambda^{2\pi /\alpha\beta}
e^{4\ii\pi^2 k /\alpha\beta}
)
}\, .
\qedhere
\end{align}
\end{proof}
The following corollary restates the result on the dimensions of Theorem \ref{maindef} in an important special case.
\begin{thc}
 Suppose that $\{A,B\}$ is an irreducible nontrivial representation of the class $\cC_0$ for relation (\ref{qplane1}). Then
 \begin{align*}
 \dim \Hh_A(\lambda,\mu) &= 0\quad {\rm for}\quad  \lambda  \notin \cS(q)^+, \quad \dim \Hh_A(\lambda,\mu) = 1\quad {\rm for} \quad  \lambda  \in \cS(q)^+, \\
 \dim \Hh_B(\lambda,\mu) &= 0\quad {\rm for} \quad  \mu  \notin \cS(q)^+,\quad \dim \Hh_B(\lambda,\mu) = 1\quad{\rm for}\quad  \mu  \in \cS(q)^+.
\end{align*}
\end{thc}
\begin{proof} 
By Corollary \ref{wellbehcor1},  $\{A,B\}$ is unitarily equivalent to a pair $\{\delta_1 e^{\alpha P},\delta_2 e^{\beta Q}\}$, where\, $\delta_1, \delta_2\in \{1,-1\}$\, and\, $\alpha \beta=\theta_0$, $|\theta_0|<\pi$. Hence the assertion follows at once from Theorem \ref{maindef}.
\end{proof}
We now  continue the  considerations towards the modified resolvent relations.
\begin{thp}\label{maphahb}
$(i)$\,  $R_{\bar\mu}(B) R_{\bar\lambda}(A) - \frac{1}{\bar\lambda\bar\mu(1-\bar q)}\, I$\, maps\, $\Hh_B(\lambda,\mu)$\, into\, $\Hh_A(\lambda,\mu)$.\\
$(ii)$\, $ R_{\bar\lambda\bar q}(A)R_{\bar\mu \bar q}(B) + \frac{1}{\bar q\bar\lambda\bar\mu(1-\bar q)}\, I$\, maps\, $\Hh_A(\lambda,\mu)$\, into\, $\Hh_B(\lambda,\mu)$.
\end{thp}
\begin{proof}
(i):  Indeed, for $\phi \in \cD_0$ we have
\[
 (B-\mu I)(A-\lambda I)\phi= \bar q(A-\lambda q I)(B-\mu q I)\phi + \lambda\mu(1-q) \phi.
\]
Therefore, for $\eta\in \Hh_B(\lambda,\mu)$ we derive
\begin{gather*}
\big\langle\big(R_{\bar\mu}(B) R_{\bar\lambda}(A) - \frac{1}{\bar\lambda\bar\mu(1-\bar q)}\, I \big)\eta, (B-\mu I)(A-\lambda I)\phi\big\rangle
\\
=
\langle R_{\bar\mu}(B) R_{\bar\lambda}(A)\eta, (B-\mu I)(A-\lambda I)\phi\rangle 
-
\frac{1}{\bar\lambda\bar\mu(1-\bar q)}\, \langle \eta, (B-\mu I)(A-\lambda I)\phi \rangle
\\
=
- 
\frac{q}{\bar\lambda\bar\mu(1-\bar q)}\, \langle \eta, (A-\lambda q I)(B-\mu q I)\phi\rangle =0,
\end{gather*}
that is, $\big(R_{\bar\mu}(B) R_{\bar\lambda}(A) - \frac{1}{\bar\lambda\bar\mu(1-\bar q)} \,I\big)\eta \in \Hh_A(\lambda,\mu)$.

(ii) is proved in a  similar manner.
\end{proof}

From Proposition \ref{maphahb} it follows that the operator
\begin{equation}\label{resolvents-p1}
\Bigl( R_{\bar\mu}(B) R_{\bar\lambda}(A) - \frac{1}{\bar\lambda\bar\mu(1-\bar q)}\, I \Bigr)
\Bigl( R_{\bar\lambda\bar q}(A)R_{\bar\mu \bar q}(B) + \frac{1}{\bar q\bar\lambda\bar\mu(1-\bar q)}\, I \Bigr)
\end{equation}
maps the subspace\, $\Hh_A(\lambda,\mu)$\, into itself and that the operator
\begin{equation}\label{resolvents-p2}
\Bigl( R_{\bar\lambda\bar q}(A)R_{\bar\mu \bar q}(B) + \frac{1}{\bar q\bar\lambda\bar\mu(1-\bar q)}\, I\Bigr)
\Bigl( R_{\bar\mu}(B) R_{\bar\lambda}(A) - \frac{1}{\bar\lambda\bar\mu(1-\bar q)}\, I \Bigr)
\end{equation}
maps\, $\Hh_B(\lambda,\mu)$\, into itself. 
In the special case when  $\Hh_A(\lambda,\mu)=\Hh_B(\lambda,\mu)=\{0\}$ we know from  Section~\ref{prelimrel1} that the corresponding operators in \eqref{resolvents-p1} and \eqref{resolvents-p2} are both equal to $-q\bar\lambda^{-2}\bar\mu^{-2}(1-\bar q)^{-2}\,I$ on the whole Hilbert space.

In the general case some lengthy but straightforward  computations using   \eqref{integral1} and the relations $e^{\alpha Q} = \mathcal F e^{-\alpha P} \mathcal F^{-1}$, $e^{\beta P} = \mathcal F e^{\beta Q} \mathcal F^{-1}$, lead to the formulas
\begin{align*}
 \Bigl(& R_{\bar\lambda\bar q}(A)R_{\bar\mu\bar q}(B)+\frac{1}{\bar\lambda\bar\mu\bar q(1-\bar q)}\, I\Bigr)\psi_j(x)\qquad
\\
&=\frac{
\bar\lambda^{2\pi(m -(1-\epsilon_1)/2) /\alpha\beta}
e^{-4\ii\pi^2(m-  (1-\epsilon_1)/2 )j  /\alpha\beta}
}{
\bar\lambda
\bar q(1-\bar q)
\bar\mu
}
\sum_{l=(1+\epsilon_2)/2}^{m-(1-\epsilon_1)/2}
\bar\lambda^{-2\pi l /\alpha\beta}
e^{4\ii\pi^2 l j  /\alpha\beta}
\eta_{l}(x),
\\
\Bigl(&R_{\bar\mu}(B)R_{\bar\lambda}(A) -\frac{1}{\bar\lambda\bar\mu(1-\bar q)}\, I\Bigr)\eta_k(x)
\\
&=\frac{
-2\pi
\bar\lambda^{ -2\pi (m-k+(\epsilon_1-1)/2)/\alpha\beta}
}{
\alpha\beta
\bar\lambda\bar\mu
(1-\bar q)
}
\sum_{j=(1-\epsilon_3)/2}^{m-(\epsilon_4+1)/2}
e^{4\ii\pi^2 j(m-k+(\epsilon_1-1)/2)/\alpha\beta}
\psi_j(x).
\end{align*}
Therefore, we finally obtain
\begin{align*}
\Bigl(R_{\bar\mu}(B)&R_{\bar\lambda}(A) -\frac{1}{\bar\lambda\bar\mu(1-\bar q)}\, I\Bigr)
\Bigl(R_{\bar\lambda\bar q}(A)R_{\bar\mu\bar q}(B)+\frac{1}{\bar\lambda\bar\mu\bar q(1-\bar q)}\, I \Bigr) \psi_j(x)
\\
&=\frac{
-2\pi
}{
\alpha\beta
\bar\lambda^2
\bar\mu^2
\bar q(1-\bar q)^2
}
\sum_{k=0}^{m-1-(\epsilon_2-\epsilon_1)/2}
\sum_{l=(1-\epsilon_3)/2}^{m-(\epsilon_4+1)/2}
e^{-4\ii\pi^2 k j  /\alpha\beta}
e^{4\ii\pi^2 lk/\alpha\beta}
\psi_l(x),
\\
\Bigl(R_{\bar\lambda\bar q}(A)&R_{\bar\mu\bar q}(B)+\frac{1}{\bar\lambda\bar\mu\bar q(1-\bar q)}\, I \Bigr) \Bigl(R_{\bar\mu}(B)R_{\bar\lambda}(A) -\frac{1}{\bar\lambda\bar\mu(1-\bar q)}\,I\Bigr)\eta_k
\\
&=\frac{
-2\pi
}{
\alpha\beta
\bar\lambda^2\bar\mu^2
\bar q
(1-\bar q)^2
}
\sum_{j=(1-\epsilon_3)/2}^{m-(\epsilon_4+1)/2}
\sum_{l=(1+\epsilon_2)/2}^{m-(1-\epsilon_1)/2}
e^{-4\ii\pi^2 j(k-l)/\alpha\beta}
\bar\lambda^{2\pi(k- l) /\alpha\beta}
\eta_{l}(x).
\end{align*}
The preceding two equations are the versions of the resolvent equations for  basis elements of the subspaces $\Hh_A(\lambda,\mu)$\, and\, $\Hh_B(\lambda,\mu)$, respectively.
\section*{Acknowledgments}
The results of this paper were obtained during research visits of the first author at Leipzig University supported by the DFG grant SCHM 1009/5-1. Excellent working conditions and warm hospitality are  acknowledged. The authors express their gratitude to Prof. Lyudmila Turowska and Prof. Andrew Bakan for helpful discussions of  topics related to this research.

\end{document}